\documentclass[11pt]{amsart}   	
\pdfoutput=1

\usepackage[margin=1in]{geometry}

\usepackage{amsmath,amssymb,amsfonts,amsthm}
\usepackage{hyperref}
\usepackage{mathtools}
\usepackage[capitalise]{cleveref}
\usepackage{tikz-cd}
\usepackage{tikz}
\usepackage{enumitem}
\usepackage{makecell}
\setcellgapes{4pt}
\usepackage{pdfpages} 
\usepackage{enumitem}
\usepackage{subcaption}
\usepackage{multicol}

\usepackage{enumitem}
\usepackage{tabstackengine}

\usepackage{comment}
\usepackage{pdflscape}

\newcommand{\add}{\mathsf{add}}

\newcommand{\Gen}{\mathsf{Gen}}

\newcommand{\ind}{\mathsf{ind}}

\newcommand{\stt}{\mathsf{s}\tau\text{-}\mathsf{tilt}}

\newcommand{\taur}{\tau\text{-}\mathsf{rigid}}

\newcommand{\mods}{\mathsf{mod}}

\newcommand{\proj}{\mathsf{proj}}

\newcommand{\A}{\mathbb{A}}
\renewcommand{\L}{\Lambda}

\newcommand{\bcomp}{\mathrm{B}}
\newcommand{\ccomp}{\mathrm{C}}

\newcommand{\Pcal}{\mathcal{P}}

\DeclareMathOperator{\Hom}{\mathrm{Hom}}
\DeclareMathOperator{\Ext}{\mathrm{Ext}}
\DeclareMathOperator{\im}{\mathrm{im}}

\DeclareMathOperator{\End}{\mathrm{End}}
\DeclareMathOperator{\soc}{\mathrm{soc}}
\DeclareMathOperator{\topp}{\mathrm{top}}

\DeclareMathOperator{\rad}{\mathrm{rad}}

\newcommand{\tauexc}[1]{\ensuremath{\tau\text{-}\mathsf{exc} #1}}
\newcommand{\tfo}[1]{\ensuremath{\mathsf{TF}\text{-} #1}}

\newcommand{\ctauexc}[1]{\ensuremath{\mathsf{c}\text{-}\tau\text{-}\mathsf{exc} #1}}

\newcommand{\freegroup}[1]{\ensuremath{\mathsf{F}_{#1}}}

\newtheorem{theorem}{Theorem}[section]
\newtheorem{corollary}[theorem]{Corollary}
\newtheorem{lemma}[theorem]{Lemma}

\newtheorem{proposition}[theorem]{Proposition}

\newtheorem{setting}[theorem]{Setting}

\theoremstyle{definition}
\newtheorem{definition}[theorem]{Definition}
\newtheorem{definitionproposition}[theorem]{Definition-Proposition}

\newtheorem{example}[theorem]{Example}

\newtheorem{remark}[theorem]{Remark}

\usepackage[style=alphabetic,sorting=nty,sortlocale = nn\_NO,maxbibnames=15]{biblatex}

\title{}

\title[A braid group action on $\tau$-exceptional sequences]{Classifying Nakayama algebras with a braid group action on $\tau$-exceptional sequences}

\author{Maximilian Kaipel}
\address{Maximilian Kaipel, Fakultät für Mathematik, Universität Bielefeld, 33501 Bielefeld, Germany}
\email{mkaipel@math.uni-bielefeld.de}

\author{Håvard U. Terland}
\address{Håvard U. Terland, Department of Mathematical Sciences, Norwegian University of Science and Technology (NTNU), 7491 Trondheim, NORWAY}
\email{havard.u.terland@ntnu.no}

\thanks{H.U.T. was supported by grant number FRINAT 301375 from the Norwegian Research Council. H.U.T. wishes to thank Aslak Buan for proposing investigating whether the mutation of $\tau$-exceptional sequences over certain Nakayama algebras respects the braid group relations and for proofreading multiple iterations of an earlier version of the paper. M.K. thanks H.U.T. for encouraging him to join the project for the second version. M.K. acknowledges support from the Deutsche Forschungsgemeinschaft (DFG, German Research Foundation) – Project ID 281071066 – TRR 191}

\begin{document}

\begin{abstract}
We characterise those basic and connected Nakayama algebras $\Lambda$ for which the mutation of $\tau$-exceptional sequences respects the braid group relations. We show that this is the case if and only if $\Lambda$ is hereditary or all indecomposable projective $\Lambda$-modules have length at least $|\Lambda|$.
	\end{abstract}

\maketitle

\section{Introduction}
Exceptional sequences for finite dimensional algebras were first studied in \cite{cbw92} and \cite{ringel_exceptional}. Crawley-Boevey showed that there is a transitive braid group action on complete exceptional sequences for hereditary algebras over algebraically closed fields, and Ringel generalized this to arbitrary hereditary artin algebras. For non-hereditary algebras, however, exceptional sequences are not as well-behaved and complete exceptional sequences might not exist. Using $\tau$-tilting theory, Buan and Marsh proposed in \cite{tauexceptional_buanmarsh} a generalization of exceptional sequences for finite dimensional algebras by introducing \textit{$\tau$-exceptional sequences}, which coincide with exceptional sequences for hereditary algebras. However, it was only recently that a mutation operation on these $\tau$-exceptional sequences was suggested: In \cite{BHM2024}, Buan, Hanson and Marsh define a mutation on $\tau$-exceptional sequences for general finite dimensional algebras, which coincides with the classical mutation for hereditary algebras. Further, Nonis shows in \cite{nonis2025} that the mutation  of $\tau$-exceptional sequences over algebras of the form $R \otimes kQ$, for $Q$ an acyclic quiver and $R$ a finite dimensional commutative local $k$-algebra, also coincides exactly with the mutation of $\tau$-exceptional sequences over $kQ$, and therefore that this mutation respects the braid group relations.

In this paper, we continue the hunt for algebras such that the mutation on $\tau$-exceptional sequences over them gives a braid group action. In particular we utilize the properties of mutation of $\tau$-exceptional sequences over Nakayama algebras developed in \cite{bkt} to connect the work of Crawley-Boevey \cite{cbw92} and Ringel \cite{ringel_exceptional} with the work of Buan, Marsh and Hanson \cite{BHM2024} by proving the following theorem.

\begin{theorem}[\cref{thm:braid_group_relations_satisfied_on_cn}]\label{thm:introthm1}
    Let $C_n$ be the Nakayama algebra given as a quotient of the $n$-cycle path algebra over a field $k$, modulo the ideal generated by paths of length $n$. Then mutation of $\tau$-exceptional sequences over $C_n$ respects the braid group relations.
\end{theorem}

The algebra $C_n$ is such that every indecomposable projective module $P$ has length equal to $n$. We leverage \cref{thm:introthm1} to show that the result remains true when considering Nakayama algebras all of whose indecomposable projective modules have length at least $n$. Our approach may be seen as an application of Drozd--Kirichenko rejection \cite{DrozdKirichenko}, see also \cite[Sec. 3]{Adachi2016}.

\begin{theorem}[\cref{thm:longerlengths}, reformulated]\label{thm:introthm2}
    Let $\Lambda$ be a basic and connected Nakayama algebra such that $|\Lambda|=n$ and $\ell(P) \geq n$ for all indecomposable projective $\Lambda$-modules $P$. There is a bijection between $\tau$-exceptional sequences in $\mods \Lambda$ and $\tau$-exceptional sequences in $\mods C_n$. Moreover, this bijection is compatible with the mutation of complete $\tau$-exceptional sequences. Consequently, the mutation of $\tau$-exceptional sequences over $\Lambda$ respects the braid group relations.
\end{theorem}

It should be noted that the conditions (M1) and (M2), see \cref{prop:properties_mutation_braid_group}, which are central to the proof of \cref{thm:introthm1}, do not generally hold for the Nakayama algebras in \cref{thm:introthm2}. Nonetheless, the poset of support $\tau$-tilting modules in $\mods \Lambda$ is isomorphic to the poset of support $\tau$-tilting modules in $\mods C_n$ by \cite[Thm. 3.11]{Adachi2016}. To complete the classification, we then show that these are essentially all Nakayama algebras for which the mutation of complete $\tau$-exceptional sequences satisfies the braid group relations. The only remaining ones are the hereditary Nakayama algebras. 

\begin{theorem}[\cref{thm:notethm1}]\label{thm:introthm3}
    Let $\Lambda$ be a basic and connected Nakayama algebra. Assume that there is an indecomposable projective $\Lambda$-module $P$ with $\ell(P) < |\Lambda|$. Then the mutation of $\tau$-exceptional sequences satisfies the braid group relations if and only if $\Lambda$ is hereditary. 
\end{theorem}

Combining the previous theorems, we obtain the following characterisation.
\begin{theorem}\label{thm:introthm4}
    Let $\Lambda$ be a basic and connected Nakayama algebra. Then the mutation of $\tau$-exceptional sequences in $\mods \Lambda$ satisfies the braid group relations if and only if $\Lambda$ is hereditary or $\ell(P) \geq |\Lambda|$ for all indecomposable projective $\Lambda$-modules.
\end{theorem}

The paper is structured as follows: We first recall the basics of $\tau$-tilting theory including $\tau$-exceptional sequences, then discuss mutation of $\tau$-exceptional sequences generally before working with a specific class of Nakayama algebras to obtain \cref{thm:introthm1} in \cref{sec:cycliccase}. In \cref{sec:longerprojs} we show that an explicit bijection between TF-ordered $\tau$-rigid modules gives coinciding mutation theories of complete $\tau$-exceptional sequences \cref{thm:introthm2}. In the final section, we show that these are the only non-hereditary examples by constructing explicit counterexamples, which proves \cref{thm:introthm3}.

\section{Background}
All algebras are in this paper assumed to be basic, finite dimensional algebras over a field $k$. Given an algebra $\Lambda$, we denote by $\mods \Lambda$ the category of finite dimensional left $\Lambda$-modules. We assume all modules to be basic, and denote by $|M|$ and $\ell(M)$ the number of indecomposable direct summands and the length of a module $M$ respectively. We let $\proj \Lambda$ denote the subcategory of projective $\Lambda$-modules. We reserve the constant $n$ to refer to the rank of an algebra $\Lambda$. 

Given a module $M$ in $\mods \Lambda$, we denote by $\add M$ the subcategory of modules which are direct sums of direct summands of $M$. We denote by $\Gen M$ the subcategory of modules $N$ such that there exists an epimorphism in $\Hom(\add M,N)$. We also have the left and right \textit{perpendicular subcategories} $M^\perp = \{X \in \mods \Lambda \mid \Hom_\Lambda(M,X) = 0\}$ and ${}^\perp M = \{X \in \mods \Lambda \mid \Hom_\Lambda(X,M) = 0\}$.

Given a subcategory $\mathcal{X}$ of $\mods \Lambda$ and a $\Lambda$-module $M$, a \textit{left $\mathcal{X}$-approximation} of $M$ is a map $f\colon M \to X$ where $X \in \mathcal{X}$ such that any map $g\colon M \to X'$ for $X' \in \mathcal{X}$ factors through $f$, i.e, there exists a morphism $h\colon X \to X'$ such that $h \circ f = g$. Right approximations are defined dually. An approximation is \textit{minimal} if the map is minimal in the usual sense.

\subsection{$\tau$-tilting theory}
Closely following \cite{tau}, we recall the basics of $\tau$-tilting theory. 

\begin{definition}\cite[Def. 0.1]{tau}
    A module $M$ is called \textit{$\tau$-rigid} if $\Hom_\Lambda(M,\tau M) = 0$, where $\tau$ denotes the Auslander-Reiten translate in $\mods \Lambda$. A $\tau$-rigid module $M$ with $|M| = n$ is called a $\tau$-tilting module.
\end{definition}

Given a $\tau$-rigid module $M$, one might ask which modules $X$ are compatible with $M$ in the sense that $M \oplus X$ is $\tau$-rigid. We are in particular interested in two dual notions of such compatible modules, namely those that maximize $\Gen (M\oplus X)$ and those that minimize it. 

\begin{definition}
    Let $M$ be a $\tau$-rigid module. Let $X \notin \add M$ be an indecomposable module such that $X\oplus M$ is $\tau$-rigid.

    \begin{itemize}
        \item We say that an indecomposable module $X$ is a direct summand of the \textit{Bongartz complement} $\bcomp(M)$ of $M$ if and only if $X$ is not generated by any $\tau$-rigid module $M\oplus N$ unless $N$ has $X$ as a summand.
        \item We say that an indecomposable module $X$ is a direct summand of the \textit{co-Bongartz complement} $\ccomp(M)$ of $M$ if and only if $X$ is generated by $M$.
    \end{itemize}
    
\end{definition}

\begin{remark}
    By \cite[Thm. 4.4]{dirrt} the above definition is equivalent to the definition of Bongartz completions appearing in \cite[Thm. 2]{tau}.
    
    If $|M| = r$ then $|\bcomp(M)| = n-r$ and $|\ccomp(M)| \leq n-r$. In particular, $M \oplus \ccomp(M)$ need not be $\tau$-tilting over $\Lambda$ but will always be a $\tau$-tilting module over $\Lambda/\Lambda e \Lambda$ for some idempotent $e$. Some authors introduce non-module (shifted projective) summands to appear in the co-Bongartz complement that convey equivalent information to the idempotent $e$ appearing above. We refer the reader to \cite[Thm. 2.10 and 2.18]{tau}, \cite[Sec. 4]{dirrt} and \cite[Sec. 2 and 4]{bkt} for details.
\end{remark}

Given a $\tau$-rigid module $M$, the category $J(M) = M^\perp \cap {}^\perp \tau M$ is called the \textit{$\tau$-perpendicular category} of $M$ \cite[Def. 3.3]{jassoreduction}. Jasso  \cite{jassoreduction} showed that $J(M) \cong \mods \Gamma$ for some algebra $\Gamma$ of rank $n-|M|$ and that there is an intimate relationship between the $\tau$-tilting theory of $\Lambda$ and that of $\Gamma$.

Using the work of Jasso, Buan and Marsh \cite{tauexceptional_buanmarsh} generalized the notion of exceptional sequences for hereditary algebras to arbitrary finite dimensional algebras.

\begin{definition}\cite[Def. 1.3]{tauexceptional_buanmarsh}
    A sequence of $\Lambda$-modules $(A_1,A_2,\dots,A_t)$ is a \textit{$\tau$-exceptional sequence }in $\mods\Lambda$ if
    
    \begin{enumerate}
        \item $A_t$ is $\tau$-rigid indecomposable in $\mods \Lambda$, and
        \item if $t > 1$, then $(A_1,A_2,\dots,A_{t-1})$ is a $\tau$-exceptional sequence in $J(A_t)$.
    \end{enumerate}
\end{definition}

We let $\tauexc{\Lambda}$ denote the set of $\tau$-exceptional sequences over $\Lambda$ and $\ctauexc{\Lambda}$ the set of \textit{complete} $\tau$-exceptional sequences, that is, those of length $n = |\Lambda|$. 

It is important to consider that $\tau$ depends on the category in which it is being computed, so when talking about $\tau$-exceptional sequences we always need to keep in mind which ambient category we are in. If, for example, $(A_1,A_2,A_3)$ is a $\tau$-exceptional sequence in $\mods \Lambda$, then $(A_1,A_2)$ is a $\tau$-exceptional sequence in $J(A_3)$. However, $(A_1,A_2)$ need not be a $\tau$-exceptional sequence in $\mods \Lambda$.

We make use of subscripts (and sometimes superscripts) when we want to make explicit in which subcategory a computation is taking place, so we could write that $A_2$ is $\tau_{J(A_3)}$-rigid or that $A_1$ is $\tau_{J_{A_3}(A_2)}$-rigid. This last example demonstrates the practicality of recursively extending the definition of $\tau$-perpendicular categories. Following \cite{BHM2024}, we thus make the following definition.

\begin{definition}
    Let $(A_1,A_2,\dots,A_t)$ be a $\tau$-exceptional sequence in some given ambient module category. Then we define \[J(A_1,A_2,\dots,A_t) = \begin{cases}
        J(A_1) & \text{if $t = 1$} \\
        J_{J(A_2,A_3,\dots,A_t)}(A_1) & \text{else.}  \end{cases}\]
\end{definition}

Note that complete $\tau$-exceptional sequences always exist: For an arbitrary module category $\mods \Lambda$ any projective module is $\tau$-rigid. In particular, we can always extend any incomplete $\tau$-exceptional sequence $(B_1,B_2,\dots,B_t)$ to a longer $\tau$-exceptional sequence $(P,B_1,\dots,B_t)$ by picking $P$ to be an indecomposable projective in $J(B_1,B_2,\dots,B_t)$.

\begin{remark}
    There is a more general notion of \textit{signed} $\tau$-exceptional sequences. For brevity, we here skip this definition and refer the interested reader to \cite{tauexceptional_buanmarsh}, \cite{BHM2024} and \cite{bkt} for details on signed $\tau$-exceptional sequences especially in the context of mutation of $\tau$-exceptional sequences.
\end{remark}

The following theorem gives a useful uniqueness-property of complete $\tau$-exceptional sequences, namely that if we know $n-1$ coordinates of a complete $\tau$-exceptional sequence, then the remaining coordinate is uniquely determined.

\begin{theorem}\cite[Thm. 8]{hanson_thomas_unique}\label{thm:tau_exceptional_missing_one}
Let $A = (A_1,A_2,\dots,A_n)$ and $B = (B_1,B_2,\dots,B_n)$ be complete $\tau$-exceptional sequences such that $A_i = B_i$ for all $1 \leq i < j$ and all $j < i \leq n$ for some integer $j$. Then $A_j = B_j$.
\end{theorem}

The following corollary will be useful to us. 

\begin{corollary}\label{cor:tau_uniqueness_jasso_preserve}
   Let $(X,C)$ be a $\tau$-exceptional sequence. Then $X$ is uniquely determined by $C$ and $J(X,C)$.
\end{corollary}

\begin{proof}
    Extend $(X,C)$ to a complete $\tau$-exceptional sequence $(A_1,A_2,\dots,A_{n-2},X,C)$ using the projective modules determined by $J(X,C)$. Then $X$ must be uniquely determined by \cref{thm:tau_exceptional_missing_one}.
\end{proof}

There is a bijection between $\tau$-exceptional sequences and certain \textit{ordered} $\tau$-rigid modules.

\begin{definition}
    Let $M = M_1 \oplus M_2 \oplus \dots \oplus M_t$ be a module with $M_i$ indecomposable for all $i$. Given a module $M$ along with a given ordered decomposition, we call $M$ an \textit{ordered} module.

    Let $M = M_1 \oplus M_2 \oplus \dots \oplus M_t$ be an ordered $\tau$-rigid module. We say that $M$ is \textit{TF-ordered} if $M_i \notin \Gen(\bigoplus_{j>i} M_j)$.
    
\end{definition}

Note that a $\tau$-rigid module $M$ may be ordered in $|M|!$ distinct ways. By abuse of notation, we use the direct sum notation as usual with the caveat that we care about the order of the direct summands. We denote by $\tfo{\Lambda}$ the set of TF-ordered $\tau$-rigid $\Lambda$-modules.

\begin{example}
    Let $\Lambda$ be a rank $2$ algebra with indecomposable projective modules $P(1)$ and $P(2)$. Then $P(1) \oplus P(2)$ and $P(2) \oplus P(1)$ are two distinct TF-ordered $\tau$-rigid modules.
\end{example}

For a $\tau$-rigid module $M$, we denote by $f_M(X) = X/t_M(X)$ the co-kernel of the trace of $M$ in $X$. Given a TF-ordered $\tau$-rigid module $M \oplus N$, we then have that $f_N(M)$ is $\tau$-rigid in $J(N)$.

\begin{remark}
    For a $\tau$-rigid module $M$, we have that $(\Gen M,M^\perp)$ is a torsion pair, inducing for any module $X$ a short exact sequence $t_M(X) \xhookrightarrow{} X \twoheadrightarrow f_M(X)$. This illustrates one of many connections between $\tau$-tilting theory and torsion theory, and we refer the reader to \cite{tau} and \cite{dirrt} for more information.
\end{remark}

\begin{proposition}\cite[Prop. 5.6]{tauexceptional_buanmarsh}\label{prop:f_m-bijection-taurigid-J(M)-rigid}
    Let $M$ be a $\tau$-rigid module. Then $f_M$ gives a bijection between indecomposable $\tau$-rigid modules $N$ such that $N \oplus M \in \tfo{\Lambda}$, and indecomposable $\tau_{J(M)}$-rigid modules.
\end{proposition}

It follows that for $M \oplus N \in \tfo{\Lambda}$ with $M \text{ and }N$ indecomposable then $(f_N(M),N)$ is a $\tau$-exceptional sequence. Indeed we have the following stronger result linking TF-ordered $\tau$-rigid modules and $\tau$-exceptional sequences.

\begin{theorem}\cite[Thm. 5.1]{mendozatreffinger_stratifyingsystems}\cite[Thm. 5.4]{tauexceptional_buanmarsh}\label{thm:MendozaTreffinger}
    There is a bijection $\Psi\colon \tfo{\Lambda} \to \tauexc{\Lambda}$ defined recursively by \[\Psi(M_1 \oplus M_2 \oplus \dots \oplus M_t) = (\Psi_{J(M_t)}(f_{M_t}(M_1),f_{M_t}(M_2),\dots,f_{M_t}(M_{t-1})),M_t)\]
    
    where we get rid of unnecessary parenthesis by identifying $((A),B) = (A,B)$.
\end{theorem}

Notice that the rightmost coordinate of a $\tau$-exceptional sequence is the same module as the rightmost summand of the corresponding TF-ordered $\tau$-rigid module. In particular, for a $\tau$-exceptional sequence $(B,C)$ we have $\Psi^{-1}(B,C) = f^{-1}_C(B) \oplus C$.

\subsection{Mutation of $\tau$-exceptional sequences}
Buan, Marsh and Hanson define in \cite{BHM2024} a \textit{mutation} of $\tau$-exceptional sequences, generalizing the mutation of exceptional sequences as defined by Crawley-Boevey in \cite{cbw92}. In this paper, we restrict to considering the mutation for \textit{$\tau$-tilting finite algebras}; algebras with a finite number of $\tau$-rigid modules up to isomorphism. Let now $\Lambda$ be a $\tau$-tilting finite algebra and $(B,C)$ a $\tau$-exceptional sequence. 

\begin{definition}\cite[Def. 3.11a]{BHM2024}
    We say that a $\tau$-exceptional sequence $(B,C)$ is left regular if $C \in \proj \Lambda$ or $C \notin \add \bcomp(f^{-1}_C(B))$.
\end{definition}

\begin{definition}\cite[Def.-Prop. 4.3]{BHM2024}\label{def:regular_mutation}
    Let $(B,C)$ be a left regular $\tau$-exceptional sequence. Let \[B' = \begin{cases}
        B & \text{if $C$ is projective} \\
        f^{-1}_C(B) & \text{else}
    \end{cases}.\] Then we define the left mutation of $(B,C)$ as the $\tau$-exceptional sequence \[\varphi(B,C) = (X,B')\] where $X$ is the unique module such that $(X,B')$ is a $\tau$-exceptional sequence and $J(B,C) = J(X,B')$.
\end{definition}

\begin{remark}
    The module $X$ in the above definition is unique by \cref{cor:tau_uniqueness_jasso_preserve} if it exists. In \cite[Def.-Prop. 4.3]{BHM2024}, an explicit formula for $X$ is given.
\end{remark}

For $\tau$-tilting finite algebras, $\tau$-exceptional sequences $(B,C)$ which are not left regular still always have a well-defined mutation $\varphi(B,C)$. We do not here give the general formula for irregular mutation, and instead refer the reader to see \cite[Sec. 3.3]{BHM2024} and \cite[Def. 3.23]{bkt} for details on mutation in the irregular case.

Mutation of $\tau$-exceptional sequences of length $>2$ is defined in two steps. First, note that since $J(B,C) = J(\varphi(B,C))$ for a $\tau$-exceptional sequence, if $(A,B,C)$ is a $\tau$-exceptional sequence then so is $(A,\varphi(B,C))$. Secondly, note that for a $\tau$-exceptional sequence $(A,B,C)$, the sequence $(A,B)$ is $\tau$-exceptional in $J(C)$. Let now $A = (A_1,A_2,\dots,A_t)$ be an arbitrary $\tau$-exceptional sequence. Following \cite[Def. 5.2]{BHM2024}, we set \[\varphi_i(A) = (A_1,A_2,\dots,A_{i-1},\varphi^{J(A_{i+2},A_{i+3},\dots,A_t)}(A_i,A_{i+1}),A_{i+2},\dots,A_t)\] for $1 \leq i \leq t-1$. The superscript on $\varphi$ indicates which subcategory the computation is done in.

\subsection{Mutation of $\tau$-exceptional sequences over Nakayama algebras}
Let $T = B \oplus C$ be a TF-ordered $\tau$-rigid module over a Nakayama algebra $\Lambda$. Then $\Psi(T)$ gives a $\tau$-exceptional sequence which may be mutated, and in \cite{bkt} explicit formulas for the TF-ordered $\tau$-rigid module $T'$ such that $\Psi(T') = \varphi(\Psi(T))$ are given. Let $\overline{\varphi}(T) = \Psi^{-1}(\varphi(\Psi(T)))$. 

To compute $\overline{\varphi}(B \oplus C)$ for Nakayama algebras, we split into different cases as follows.

\vspace{2ex}
\noindent\begin{minipage}{\textwidth}
\begin{multicols}{2} 
\begin{enumerate}[leftmargin=2cm]
    \item[(TF-1)] $C \in \proj \L$;
    \begin{enumerate}
        \item[(TF-1a)] and $\Hom(C,B) = 0$;
        \item[(TF-1b)] and $\Hom(C,B) \neq 0$;
    \end{enumerate}
    \bigskip
    \item[(TF-3)] $C \not \in \add (\bcomp(B) \oplus \ccomp(B))$;
\columnbreak
    \item[(TF-2)] $C \in \add\ccomp(B)$;
    \begin{enumerate}
        \item[(TF-2a)] and $B \not \in \proj \L$;
        \item[(TF-2b)] and $B \in \proj \L$;
    \end{enumerate}
    \bigskip
    \item[(TF-4)] $C \in \add \bcomp(B) \setminus \proj \L$.
\end{enumerate}
\end{multicols}
\end{minipage}
\vspace{2ex}

TF-4 captures the irregular case.

\begin{theorem}\cite[Thm. 1.3]{bkt}\label{thm:bkt_tf_formula}
    Let $\Lambda$ be a Nakayama algebra and let $B \oplus C$ be a TF-ordered $\tau$-rigid module. Then the left mutation of TF-ordered modules is given by
    \[ \overline{\varphi}(B \oplus C) = \begin{cases}
        C \oplus B & \text{ in Case TF-1a and Case TF-3,}\\
        B \oplus f_C B & \text{ in Case TF-1b and Case TF-4,}\\
        \rad^k B \oplus B & \text{ in Case TF-2a,} \\ 
        P(\rad^k B) \oplus B & \text{ in Case TF-2b,}\\
    \end{cases}\]
    where $C \cong B/\rad^k B$ for some $k \in \{1, \dots, \ell(B)\}$ in Case TF-2 and where $P(\rad^k B) \to \rad^k B$ is the projective cover of $\rad^k B$.
\end{theorem}

\section{Mutation of $\tau$-exceptional sequences as group actions}

Let $\Lambda$ be an arbitrary $\tau$-tilting finite algebra and recall that $\ctauexc{\Lambda}$ denotes the set of complete $\tau$-exceptional sequences over $\Lambda$. Since $\Lambda$ is assumed to be $\tau$-tilting finite, $\varphi_i(A)$ (left mutation) and $\varphi_i^{-1}(A)$ (right mutation) is well-defined for all $A \in \ctauexc{\Lambda}$ and $i \in \{1,2,\dots,n-1\}$. Letting \freegroup{n-1} be the free group with generators $\{\rho_1,\rho_2,\dots,\rho_{n-1}\}$, we have a left \freegroup{n-1} action on $\ctauexc{\Lambda}$ given by setting $\rho_i \cdot A = \varphi_i(A)$. 

\begin{definition}\label{def:braidgroup_relations}
    By the braid group relations on $n$ strands we mean the relations B1 and B2 below, where we assume $i,j \in \{1,2,\dots,n-1\}$: \begin{itemize}
         \item[\quad B1:] $\rho_i\rho_j = \rho_j\rho_i$ for $|i-j|\geq 2$.
        \item[\quad B2:] $\rho_{i}\rho_{i+1}\rho_{i} = \rho_{i+1}\rho_{i}\rho_{i+1}$ for $i\leq n-2$.
    \end{itemize}

    The braid group $B_n$ is defined as the group given by the generators $\{ \rho_1,\rho_2,\dots,\rho_{n-1} \}$ and the braid group relations B1 $\cup$ B2. 

\end{definition}

\begin{proposition}\label{prop:b1_relations_holds}
    The relations B1 are always respected by the canonical $\freegroup{n-1}$-action on $\ctauexc{\Lambda}$ in the sense that $\rho_i\rho_j\cdot A = \rho_j \rho_i \cdot A$ for any $A \in \ctauexc{\Lambda}$ and $i,j \in \{1,2,\dots,n-1\}$ where $|i-j|\geq 2$.
\end{proposition}

\begin{proof}
        Without loss of generality, assume $i+1 < j$. Let $A = (A_1,A_2,\dots,A_n)$ be a complete $\tau$-exceptional sequence over $\Lambda$. To compute $\varphi_k(A)$ we need two pieces of information: The $\tau$-exceptional sequence $A_{(k,k+1)} = (A_k,A_{k+1})$ and the ambient category $J(A_{k+2},\dots,A_n)$. 
        
        We need to check that $\varphi_i(\varphi_j(A))$ agrees with $\varphi_j(\varphi_i(A))$ in coordinates $i,i+1$ and $j,j+1$ as all other coordinates will be left unchanged with respect to $A$. 

        We have \[\varphi_j(\varphi_i(A))_{(i,i+1)} = \varphi_i(A)_{(i,i+1)} = \varphi_i(\varphi_j(A))_{(i,i+1)}\] where the first equality follows from noting that $\varphi_j$ leaves coordinates $i,i+1$ unchanged, and the second equality comes from the fact that \[J(A_{i+1},\dots,A_j,A_{j+1},\dots,A_n) = J(A_{i+1},\dots,\varphi_j(A)_j,\varphi_j(A)_{j+1},\dots,A_n).\] 

        Further, we have \[\varphi_j(\varphi_i(A))_{(j,j+1)} = \varphi_j(A)_{(j,j+1)} = \varphi_i(\varphi_j(A))_{(j,j+1)}\] where the second equality follows from the fact that $\varphi_i$ leaves coordinates $j,j+1$ unchanged and the first equality follows from noting that $\varphi_i(A)_{k} = A_{k}$ for all $k \geq i+2$.  
\end{proof}

In the remainder of this section we will discuss a sufficient (but not necessary) criterion for the B2 relations to be respected by the mutation action on complete $\tau$-exceptional sequences.

\begin{definition}
    We say that a class $\mathcal{A}$ of algebras is closed under Jasso reductions if for any $\Lambda$  in $\mathcal{A}$ and $X$ any $\tau$-rigid module over $\Lambda$, we have that $J(X)$ is equivalent to $\mods \Gamma$ for some $\Gamma$ in $\mathcal{A}$. 
\end{definition}

The following lemmas will be useful.

\begin{lemma}\label{lem:closed_under_jasso_when_closed_under_indec_jasso}
    A class of algebras $\mathcal{A}$ is closed under Jasso reductions if and only if it is closed under \textit{indecomposable} Jasso reductions, that is requiring that if $X$ is $\tau$-rigid indecomposable in $\Lambda \in \mathcal{A}$, then $J(X) \cong \mods \Gamma$ for some $\Gamma \in \mathcal{A}$.
\end{lemma}

\begin{proof}
    Clearly being closed under Jasso reductions imply being closed under indecomposable Jasso reductions.
    
    Assume that $\mathcal{A}$ is closed under indecomposable Jasso reductions. Let $X$ be a $\tau$-rigid $\Lambda$-module for some $\Lambda \in \mathcal{A}$. We show by induction on $k = |X|$ that $J(X)$ is equivalent to a category $\mods \Gamma$ with $\Gamma \in \mathcal{A}$. For $k = 1$, we have by assumption that $J(X) \cong \mods \Gamma $ for some $\Gamma \in \mathcal{A}$. Assume now that for any algebra $\Lambda \in \mathcal{A}$ and any $\tau$-rigid $\Lambda$-module $X$ with $|X| = k$, we have $J(X) \cong \mods \Gamma$ for some $\Gamma$ in $\mathcal{A}$. Let $X \oplus Y$ be any $\tau$-rigid $\Lambda$-module with $k+1$ summands such that $Y$ is indecomposable. By \cite[Thm. 6.3]{buan2023perpendicular}, we have $J(X \oplus Y) = J_{J(X)}(Y')$ for some $Y'$ in $J(X)$. By our induction hypothesis, $J(X) \cong \mods \Gamma $ for $\Gamma  \in \mathcal{A}$, so by assumption $J_\Gamma (Y') \cong \mods \Delta$ for some $\Delta \in \mathcal{A}$. This finishes the proof. 
\end{proof}

\begin{lemma}\label{lem:disconnect_mutation_swapping}
    Let $(B,C)$ be a $\tau$-exceptional sequence in a disconnected module category of which $B$ and $C$ belong to different components. Then $\varphi(B,C) = (C,B)$.
\end{lemma}

\begin{proof}
    Note that $B \oplus C = \psi^{-1}(B,C)$ and $C \oplus B = \psi^{-1}(C,B)$, so $f_C^{-1}(B) = B$, and $(B,C)$ is left regular because if $C$ is in the Bongartz complement of $B$ then $C$ must be projective by the characterization of Bongartz complements in disconnected module categories given in \cite[Lem. 4.12]{bkt}. Since \[J(B,C) = J(B \oplus C) = J(C \oplus B) = J(C,B)\] where the first and last equalities follow from \cite[Thm. 6.4]{buan2023perpendicular}, we can by \cref{def:regular_mutation} conclude that $\varphi(B,C) = (C,B)$.
\end{proof}

Inspired by the proof strategy of Crawley-Boevey in \cite{cbw92}, where it is demonstrated that the braid group relations are satisfied for mutation of exceptional sequences over hereditary algebras, we now give sufficient criteria for the braid group relations being satisfied for mutation of $\tau$-exceptional sequences over algebras in a belonging to a class closed under Jasso reductions.

\begin{proposition}\label{prop:properties_mutation_braid_group}

    Consider the two criteria below for an algebra $\Lambda$.

    \begin{enumerate}
        \item[M1:] For any $\tau$-exceptional sequence $(B,C)$ we have that $\varphi(B,C) = (?,B)$.
        \item[M2:] For any indecomposable $\tau$-rigid module $X$ over $\Lambda$ and any $(B,C)$ which is $\tau$-exceptional both in $\mods \Lambda$ and in $J(X)$, we have $\varphi^{J(X)}(B,C) = \varphi(B,C)$.
    \end{enumerate}

    Let now $\mathcal{A}$ be a class of algebras closed under Jasso reductions such that criteria M1 and M2 above hold for any algebra in $\mathcal{A}$. Then for any algebra in this class, the mutation of complete $\tau$-exceptional sequences respects the braid group relations.
\end{proposition}

\begin{proof}
    Let $(A_1,A_2,\dots,A_n)$ be any complete $\tau$-exceptional sequence in $\ctauexc{\Lambda}$ for some $\Lambda$ in $\mathcal{A}$. We can assume that $t\geq 3$.

    Let $1 \leq i \leq t-2$ and consider the sequence $(A_i,A_{i+1},A_{i+2})$ in $J(A_{i+3},\dots,A_n)$. Since $\mathcal{A}$ is closed under Jasso reductions, we can consider $(A_i,A_{i+1},A_{i+2})$ to be a $\tau$-exceptional sequence in $\mods \Gamma$ for an algebra $\Gamma$ in $\mathcal{A}$. 

    We compute \begin{align}
        \varphi_1\circ \varphi_2 \circ \varphi_1(A_i,A_{i+1},A_{i+2}) &= \varphi_1 \circ \varphi_2(A^\star_{i+1},A_i,A_{i+2}) & \text{(By M1 in $J_\Gamma(A_{i+2})$)} \label{eq:line1_braid_relation}\\
        &= \varphi_1(A^\star_{i+1},A^\star_{i+2},A_i) & \text{(By M1 in $\mods \Gamma$)}\\
        &= (A^{\star\star}_{i+2},A^\star_{i+1},A_i) & \text{(By M1 in $J_\Gamma(A_{i})$)} \\ \notag \\
        \varphi_2\circ \varphi_1 \circ \varphi_2(A_i,A_{i+1},A_{i+2}) &= \varphi_2 \circ \varphi_1(A_{i},A^\dagger_{i+2},A_{i+1}) & \text{(By M1 in $\mods \Gamma$)}\\
        &= \varphi_2(A^{\dagger\dagger}_{i+2},A_{i},A_{i+1}) & \text{(By M1 in $J_\Gamma(A_{i+1})$)}\\
        &= (A^{\dagger\dagger}_{i+2},A^\star_{i+1},A_i) & \text{(By M2 and \cref{eq:line1_braid_relation})}
    \end{align}

    Lastly by \cref{cor:tau_uniqueness_jasso_preserve} applied to $(A^{\dagger\dagger}_{i+2},A^\star_{i+1})$ and $(A^{\star\star}_{i+2},A^\star_{i+1})$ in $J(A_i)$ we see that $A^{\dagger\dagger}_{i+2} = A^{\star\star}_{i+2}$. We thus have that \[\varphi_i\circ \varphi_{i+1} \circ \varphi_i(A_1,A_2,\dots,A_n) = \varphi_{i+1} \circ \varphi_i \circ \varphi_{i+1}(A_1,A_2,\dots,A_n),\] demonstrating that mutation on $\ctauexc{\Lambda}$ respects the B2 relations in \cref{def:braidgroup_relations}. Since the mutation always respects the B1 relations by \cref{prop:b1_relations_holds}, the braid group relations on $n$ strands are satisfied.
\end{proof}

Given that the properties discussed above are inspired by Crawley-Bovey's work on exceptional sequences, the following proposition is no surprise.

\begin{proposition}\label{prop:hereditary_properties_mutation_braid_group}
    Properties M1 and M2 in \cref{prop:properties_mutation_braid_group} hold for hereditary algebras.
\end{proposition}

\begin{proof}
    By \cite[Thm. 6.1]{BHM2024} mutation of $\tau$-exceptional sequences coincide with mutation of exceptional sequences for hereditary algebras. 

    Further, by the left mutation property in \cite[Lem. 6]{cbw92}, property M1 in \cref{prop:properties_mutation_braid_group} hold for all hereditary algebras.

    Since mutation of exceptional sequences for hereditary algebras is computed pair-wise, property M2 in \cref{prop:properties_mutation_braid_group} holds automatically for hereditary algebras. To be precise, let $(B,C,X)$ and $(B,C)$ be $\tau$-exceptional sequences over an hereditary algebra. Thus they are both exceptional sequences. By the definition of mutation of exceptional sequences for hereditary algebras, see \cite[Lem. 6]{cbw92} and \cite[Sec. 1.1]{BRT2011}, we have $\varphi_1(B,C,X) = (C',B,X)$ where by definition $(C',B) = \varphi(B,C)$. Thus $\varphi^{J(X)}(B,C) = (C',B) = \varphi(B,C)$ as wanted.
\end{proof}

\section{Braid group relations for Nakayama algebras of the form $C_n$}\label{sec:cycliccase}

Consider the quivers below.
\begin{equation*} \label{eq:Nakayamaquiver}
\begin{tikzcd}[column sep=5mm, row sep=3mm]
    &&&&&&&& 2 \arrow[ld] & 3 \arrow[l] \\
    \overrightarrow{\A}_n: &n-1 \arrow[r] & n-2 \arrow[r] &\dots \arrow[r] & 1 \arrow[r] &0, & \overrightarrow{\Delta}_n: &1 \arrow[rd] & && \vdots \arrow[lu]\\
    &&&&&&&& 0 \arrow[r] & n-1 \arrow[ru]
\end{tikzcd}\end{equation*} Let $A_n = k\overrightarrow{\A}_n$ and $C_n = k\overrightarrow{\Delta}_n/r^n$. These algebras are both Nakayama algebras and $A_n$ plus $C_1 = A_1$ are also hereditary.

\begin{definitionproposition}
    
   Let $\mathcal{N}$ be the set of algebras of the form \[A_{a_1} \times \dots \times A_{a_i} \times C_{c_1} \times \dots \times C_{c_j}\] for integers $ i,j \geq 0$ and integers $a_k,c_{k'} \geq 1$ for $1 \leq k \leq i$ and $1 \leq k' \leq j$. Then $\mathcal{N}$ is closed under Jasso reductions.

\end{definitionproposition}

\begin{proof}
    Let $\Lambda = \Gamma \times \Delta$ be a disconnected algebra, and let $X$ be indecomposable $\tau$-rigid over $\Lambda$. Then $X$ is either naturally a $\Gamma$ or a $\Delta$-module. Assume without loss of generality that $X$ is naturally a $\Gamma$-module. Then $J(X) \cong J_{\mods \Gamma}(X) \times \mods \Delta$.

    By the above and \cref{lem:closed_under_jasso_when_closed_under_indec_jasso} it is thus sufficient to show that for a connected algebra $\Lambda$ in $\mathcal{N}$ and an indecomposable $\tau$-rigid module $X$ over $\Lambda$, we have $J(X) \cong \mods \Lambda'$ for some $\Lambda' \in \mathcal{N}$. Assume first that $\Lambda = A_n$. Let $X$ be indecomposable $\tau$-rigid. By \cite[Lem. 2]{MustafaObaid2013} we have $J(X) \cong \mods(A_{i} \times A_{n-i-1})$ for some $i \in \{0,1,2,\dots,n-1\}$. Assume now that $\Lambda = C_n$ and let $X$ be indecomposable $\tau$-rigid over $\Lambda$. Then by \cite[Prop 6.5]{Msapato2022}, $J(X) \cong \mods( A_i \times C_j)$ for $i = \ell(X)-1$ and $j = n-i-1$. This finishes the proof.
\end{proof}

It is well-known that any indecomposable module over a Nakayama algebra is uniquely determined by its top and its length $l$. We let $M^t_l$ denote the module with top $S(t \text{ mod } n)$ and length $l \leq n$ over the algebra $C_n$, where $S(i)$ denotes the simple module with projective cover $P(i)$ for $0 \leq i < n$. To simplify notation, we denote by $S(i)$ and $P(i)$ the modules $S(i \text{ mod } n)$ and $P(i\text{ mod }  n)$ when working over $C_n$.

\begin{proposition}\label{prop:nakayama_cn_properties}
    
    \begin{enumerate}
        \item Let $\Lambda \in \mathcal{N}$. Any indecomposable $\Lambda$-module is $\tau$-rigid.
        \item Let $\Lambda \in \mathcal{N}$ and let $X$ be indecomposable $\Lambda$-module generated by $M$ for some not necessarily indecomposable module $M$. Then there is an indecomposable summand of $M$ generating $X$. Dually, if $X$ is cogenerated by $N$, there is an indecomposable summand of $N$ that also cogenerates $X$.
        \item For a non-projective module $M^t_l$, we have $\tau M^t_l = M^{t-1}_l$ and $\rad^i (\tau M^t_l)= \tau(\rad^i M^t_l)$.
        \item Let $M$ be a non-projective indecomposable $C_n$-module. Let $N$ be a direct summand of the Bongartz complement of $M$. Then $\Hom_{C_n}(M,N) = 0$.
        \item Let $B \oplus C$ be a TF-ordered $\tau$-rigid $C_n$-module, so that $C$ is not a direct summand of the Bongartz complement of $B$. Then $\Hom_{C_n}(C,B) = 0$.
    \end{enumerate}
\end{proposition}

\begin{proof}
    \begin{enumerate}
        \item This is well known and follows e.g from \cite[Prop. 2.5]{Adachi2016}, noting that the indecomposable modules of maximal length are projectives, so also $\tau$-rigid.
        \item Assume that $M$ decomposes $M = M_1 \oplus M_2$ for $M_1 \neq 0 \neq M_2$. Let $M_1 \oplus M_2 \xrightarrow[]{f = (f_1 \; f_2)} X$ be an epimorphism. Then $\im f = \im f_1 + \im f_2$. By uniseriality, $\im f_1 \subseteq \im f_2$ or $\im f_2 \subseteq \im f_1$. Assume without loss of generality that $\im f_2 \subseteq \im f_1$. Then $M_1$ generates $X$. If $M_1$ decomposes, we can continue this process until we identify an indecomposable summand of $M$ generating $M$. Since $\mathcal{N}$ is closed under taking opposite algebras, we immediately get the dual result.
        \item This formula for $\tau$ over Nakayama algebras is well known, see e.g \cite[Thm. V.4.1]{assem_skowronski_simson_2006}. The second claim follows from noting that $\rad M^t_l = M^{t-1}_{l-1}$.
        \item Assume that $M$ has top $S(t)$ and length $l$. By \cite[Prop. 4.9]{bkt} the Bongartz complement of $M$ is given by \[\bigoplus_{i = 1}^{l-1} \rad^i M \oplus \bigoplus_{i = 0}^{n-l-1} P(t+i).\]

        Assume first that $N$ is non-projective. Then $N = \rad^i M$ for some $i\geq 1$. Thus the top of $M$ does not appear as a composition factor of $N$, and $\Hom_{C_n}(M,N) = 0$.

        Assume now that $N$ is projective. We have a map from $M$ to $P(i)$ exactly when $\soc(P(i)) = S(i+1)$ is a composition factor of $M$. The composition factors of $M$ are $S(t),S(t-1),\dots, S(t-l-1)$. The projectives with the aforementioned modules as socle are the projective $P(t-1),P(t-2),\dots,P(t-l-2)$, which are exactly the projectives not appearing as summands of the Bongartz complement of $M$. Thus $\Hom(M,N) = 0$.
        \item Note first that $C$ cannot be projective, as otherwise it would be in the Bongartz complement of $B$. If $B$ is projective, then $B$ is a summand of the Bongartz complement of $C$ and (4) gives that $\Hom_{C_n}(C,B) = 0$ as wanted. If $B$ is not projective, then by \cite[Lem. 5.4]{bkt} every map $f\colon C \to B$ is epi or mono. There can not be any epi map from $C$ to $B$ as $B \oplus C$ is TF-ordered. If there is a mono map from $C$ to $B$ then we have $C \cong \rad^iB$ implying that $C$ would be in the Bongartz complement of $B$, which we by assumption cannot have. Therefore, $\Hom_{C_n}(C,B) = 0$.
    \end{enumerate}
\end{proof}

We are now ready to work with mutation of $\tau$-exceptional sequences over algebras of the form $C_n$. We will show that criteria M1 and M2 in \cref{prop:properties_mutation_braid_group} hold for all algebras of the form $C_n$. Following the work by Crawley-Boevey \cite{cbw92} and Ringel \cite{ringel_exceptional} on mutation of exceptional sequences for hereditary algebras we will then be able to conclude that these two criteria hold for all algebras in $\mathcal{N}$, and that therefore the braid group relations hold for mutation of $\tau$-exceptional sequences for algebras in $\mathcal{N}$ and in particular for algebras of the form $C_n$.

In the rest of this section, modules are assumed to be $C_n$-modules unless it is explicitly mentioned that a more general result is given, and $X$ always denotes an indecomposable module.

\begin{lemma}\label{lem:mutationCprojectiveorleftirregular}
    Let $(B,C)$ be a $\tau$-exceptional sequence such that $C$ is projective, or $(B,C)$ is (left) irregular. Then $\varphi(B,C) = (?,B)$.
\end{lemma}

\begin{proof}
    If $C$ is projective, then $\varphi(B,C)$ is of the form $(?,B)$ independently of any assumptions on the underlying algebra by \cref{def:regular_mutation}. If $(B,C)$ is irregular, it follows from the formula in \cref{thm:bkt_tf_formula} that $\varphi(B,C)$ is of the form $(?,B)$.
\end{proof}

\begin{lemma} \label{prop:property_1_c_n}
    Property M1 in \cref{prop:properties_mutation_braid_group} holds for $C_n$, i.e, for any $\tau$-exceptional sequence $(B,C)$ we have $\varphi(B,C) = (?,B)$.
\end{lemma}

\begin{proof}
    In \cref{lem:mutationCprojectiveorleftirregular} we tackled the case where $C$ is projective or $(B,C)$ is (left) irregular, coinciding with cases TF-1 and TF-4 as in \cref{thm:bkt_tf_formula}. Assume therefore that we are in one of the two other cases, TF-2 or TF-3.

    Let $B' \oplus C$ be the TF-ordered module such that $f_C(B') = B$. We show that $B' = B$:

    \begin{itemize}
        \item If we are in case TF-2 then $C$ is generated by $B'$ and so $\text{top } C = \text{top } B$. Since $C \ncong B$ we must have $\ell(C) > \ell(B)$. Since top $C$ appears only once in the composition factor of $C$, we have $\Hom_{C_n}(B,C) = 0$.
        
        \item If we are in TF-3 then $C$ is not in the Bongartz complement of $B'$ and by \cref{prop:nakayama_cn_properties}(5) we have that $\Hom_{C_n}(C,B') = 0$ so $f_C(B') = B'$, so $B' = B$.
    \end{itemize}

    Since $B' = B$, \cref{thm:bkt_tf_formula} gives that $\varphi(B,C)$ is of the form $(?,B)$.
\end{proof}

\begin{proposition}\label{prop:property_1_n}
    Property M1 in \cref{prop:properties_mutation_braid_group} holds for any algebra $\Lambda$ in $\mathcal{N}$.
\end{proposition}

\begin{proof}
    Let $(B,C)$ be a $\tau$-exceptional sequence over $\Lambda$. If $B$ and $C$ belong to different components of the module category of $\Lambda$, then by \cref{lem:disconnect_mutation_swapping} we have $\varphi(B,C) = (C,B)$ which is of the form $(?,B)$ as wanted.

    If $B \text{ and }C$ are in the same component of $\mods \Lambda$, then this component is the module category of an algebra of the form $A_n$ or $C_n$. For hereditary algebras, we always have \cref{prop:hereditary_properties_mutation_braid_group} that $\varphi(B,C) = (?,B)$ and by \cref{prop:property_1_c_n} the property also holds for $C_n$.
\end{proof}

We will now focus on proving that property M2 in \cref{prop:properties_mutation_braid_group} also holds for any algebra in $\mathcal{N}$. An important ingredient of our proof is that $\tau_{J(X)}$-rigidity in $J(X)$ coincides with $\tau$-rigidity in $\mods C_n$. In order to establish this key fact, we need some lemmas.

\begin{lemma}\cite[Prop. 6.4, Prop. 6.5]{Msapato2022}\label{lem:projectives_in_jasso_reduction_c_n}
    Let $X$ have length $k$ in $\mods C_n$. Then $J(X) \cong \mods A_{k-1} \times \mods C_{n-k}$. Further, the projectives living in $\mods C_{n-k}$ are exactly the projective summands of the Bongartz complement of $X$. The non-projective summands of the Bongartz complement of $X$ give the projectives living in $\mods A_{k-1}$.
\end{lemma}

In the context of the lemma above, we call $\mods A_{k-1} \subseteq J(X)$ the \textit{hereditary} component of $J(X)$ and $\mods C_{n-k} \subseteq J(X)$ the \textit{cyclic component} of $J(X)$. See \cref{fig:tau_perpendicular_in_cn} for an example.

\usetikzlibrary{calc,trees}
\newcommand{\convexpath}[2]{
[   
    create hullnodes/.code={
        \global\edef\namelist{#1}
        \foreach [count=\counter] \nodename in \namelist {
            \global\edef\numberofnodes{\counter}
            \node at (\nodename) [draw=none,name=hullnode\counter] {};
        }
        \node at (hullnode\numberofnodes) [name=hullnode0,draw=none] {};
        \pgfmathtruncatemacro\lastnumber{\numberofnodes+1}
        \node at (hullnode1) [name=hullnode\lastnumber,draw=none] {};
    },
    create hullnodes
]
($(hullnode1)!#2!-90:(hullnode0)$)
\foreach [
    evaluate=\currentnode as \previousnode using \currentnode-1,
    evaluate=\currentnode as \nextnode using \currentnode+1
    ] \currentnode in {1,...,\numberofnodes} {
-- ($(hullnode\currentnode)!#2!-90:(hullnode\previousnode)$)
  let \p1 = ($(hullnode\currentnode)!#2!-90:(hullnode\previousnode) - (hullnode\currentnode)$),
    \n1 = {atan2(\y1,\x1)},
    \p2 = ($(hullnode\currentnode)!#2!90:(hullnode\nextnode) - (hullnode\currentnode)$),
    \n2 = {atan2(\y2,\x2)},
    \n{delta} = {-Mod(\n1-\n2,360)}
  in 
    {arc [start angle=\n1, delta angle=\n{delta}, radius=#2]}
}
-- cycle
}

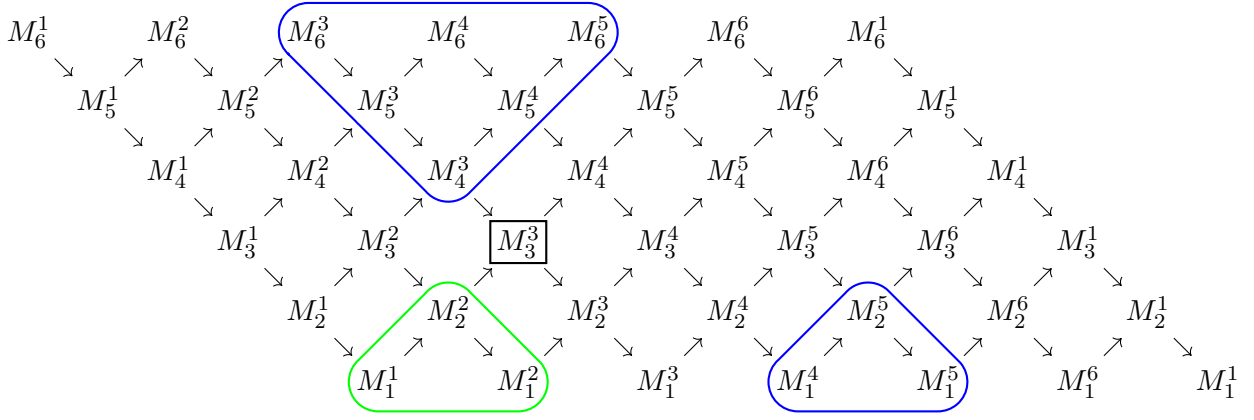
\begin{figure}
    \centering

    \begin{tikzpicture}[scale = 0.93]

    \draw[black!20!black,thick]
            (9.4,3.3) rectangle (8.6,2.7);

        \foreach \i in {1,...,6}
            \foreach \j in {1,...,7}
                \pgfmathtruncatemacro\k{mod(\j-1,6)+1}
                \node (\j\i) at (2*\j+6-\i,\i) {$M^{\k}_{\i}$};
        
        \foreach\i in {1,...,7}
            \foreach\j in {2,...,6}
                \pgfmathtruncatemacro\k{\j-1}
                \draw (\i\j) edge[->] (\i\k);

        \foreach\i in {1,...,5}
            \foreach\j in {1,...,6}
                \pgfmathtruncatemacro\k{\j+1}
                \pgfmathtruncatemacro\t{\i+1}
                \draw (\j\i) edge[->] (\k\t);

        \draw[blue,thick] \convexpath{36,46,56,34,35}{12pt};
        \draw[blue,thick] \convexpath{41,52,51}{12pt};
        \draw[green,thick] \convexpath{11,22,21}{12pt};

    \end{tikzpicture}
    
    \caption{The AR-quiver of $C_6$ with the subcategory $J(M^3_3)$ of $\mods C_6$ highlighted in blue and green. Notice that the modules with top $S(1)$ are drawn twice. The hereditary component of $J(M^3_3)$ is illustrated by the green boundary and the cyclic component is illustrated by the blue boundary.}
    \label{fig:tau_perpendicular_in_cn}
\end{figure}

\begin{lemma}\label{lem:m_notinjasso_but_generated_by_jasso_characterization}
    Let $\Lambda$ be an arbitrary algebra and $M$ an indecomposable $\Lambda$-module and $X$ a $\tau$-rigid $\Lambda$-module. If $M \notin J(X)$ is generated by some module in $J(X)$, then $\Hom_{\Lambda}(X,M) \neq 0$.
\end{lemma}

\begin{proof}
    Let $N \twoheadrightarrow M$ be an epimorphism with $N \in J(X)$. Since $N \in J(X)$ we have $\Hom_{\Lambda}(N,\tau X) = 0$ and therefore $\Hom_{\Lambda}(M,\tau X) = 0$. But since $M \notin J(X)$ we must have $\Hom_{\Lambda}(X,M) \neq 0$
\end{proof}

The lemma below is useful for determining if modules are $\tau$-rigid in a given $\tau$-perpendicular subcategory.

\begin{lemma}\cite[Prop 5.8]{auslandersmalo81}\label{lem:ext_characterization_of_maps_to_tau}
    Let $\Lambda$ be an arbitrary algebra and $M,N$ two $\Lambda$-modules. We have $\Hom(M,\tau N) = 0$ if and only if $\Ext^1_\Lambda(N,\Gen M) = 0$. Further, if $J(X)$ is a $\tau$-perpendicular subcategory of $\mods \Lambda$, then $\Hom(M,\tau_{J(X)}N) = 0$ if and only if $\Ext^1_\Lambda(N,\Gen_\Lambda M \cap J(X)) = 0$.
\end{lemma}

\begin{proof}
     Since $J(X)$ is closed under extensions \cite[Prop. 3.6, Thm. 3.8]{jassoreduction}, we get the result by applying \cite[Prop 5.8]{auslandersmalo81} in $J(X)$.
\end{proof}

The following technical lemma will be used to prove \cref{lem:tau_rigid_in_J(x)_iff_tau_rigid}(3).

\begin{lemma}\label{lem:minimal_length_generator_in_cyclic_component}
Let $X$ be an indecomposable $C_n$-module and let $U \notin J(X)$ be generated by a module in the cyclic component of $J(X)$. Let $U'$ be the module of least length in the cyclic component of $J(X)$ generating $U$ and let $V'$ be a submodule of $U'$ also in the cyclic component of $J(X)$. For any module $V \notin J(X)$ generated by $V'$, we then have that $V'$ is the module in the cyclic component of $J(X)$ of least length generating $V$.
\end{lemma}

\begin{proof}
    We will need some observations on how the cyclic component of $J(X)$ embeds in the module category $\mods C_n$. Note that an indecomposable module $Y$ in $\mods C_n$ lies in the cyclic component of $J(X)$ if and only if it is generated by some projective in the cyclic component and lies in $X^\perp$. This is because the condition of $Y$ lying in ${}^\perp \tau X$ is implied by $Y$ being generated by a module in $J(X)$. Let now $S(x)$ be the top of $X$. For an indecomposable module $Y$ in $\mods C_n$, let \[\ell^{>x}(Y) = \ell(t_{P(x)}(Y))\] where $t_{P(x)}(Y)$ denotes the trace of $P(x)$ in $Y$. Note that $Y$ lies in $X^\perp$ exactly when $\ell^{>x}(Y) \notin \{1,2,\dots,\ell(X)\}$. If $\ell^{>x}(Y) > 0$ and $Y$ embeds in some module $Y'$ then $\ell^{>x}(Y') = \ell^{>x}(Y)$. Also, if $\ell^{>t}(M^t_l) > 0$ then $\ell^{>t}(M^t_{l+1}) = \ell^{>t}(M^t_l)+1$.

    If $Y \notin J(X)$ and $Y$ is generated by some module in the cyclic component of $J(X)$, then $Y \notin X^\perp$ and so $\ell(X) \geq \ell^{>x}(Y) > 0$. So for any module $M$ in the cyclic component of $J(X)$ generating $Y$, we have $\ell^{>x}(M) \geq \ell(X)+1$. Note that $M$ is of minimal length among the modules in the cyclic component of $J(X)$ generating $Y$ exactly when $\ell^{>x}(M) = \ell(X)+1$.

    We can now conclude. Since $V'$ generates $V \notin J(X)$ and $V'$ lies in the cyclic component of $J(X)$, we have $\ell^{>x}(V') > 0$. We also have that $V'$ injects into $U'$, so $\ell^{>x}(V) = \ell^{>x}(U') = \ell(X) +1$. Thus $V'$ is the module in the cyclic component $J(X)$ of minimal length generating $V$, as wanted.
\end{proof}

\begin{lemma}\label{lem:tau_rigid_in_J(x)_iff_tau_rigid}
    Let $M,N$ be indecomposable modules in $J(X)$ for $X$ an indecomposable $C_n$-module. Then $\Hom(M,\tau N) = 0$ if and only if $\Hom(M,\tau_{J(X)}(N)) = 0$.
\end{lemma}

\begin{proof}
    By \cref{lem:ext_characterization_of_maps_to_tau}, $\Hom(M,\tau N) = 0$ implies $\Hom(M,\tau_{J(X)} N) = 0$. We will now prove the converse by splitting into three cases depending on which component of $J(X) \cong \mods( A_i \times C_j)$ the modules $M$ and $N$ lie in. We call $\mods A_i$ the hereditary component of $J(X)$, and $\mods C_j$ the cyclic component of $J(X)$. 

    Note first that if $X$ is projective, then we have no non-zero cyclic component, and so $M$ and $N$ both lie in the hereditary component. In the cases where there is a non-zero cyclic component, we can assume $X$ to be non-projective.

    \begin{enumerate}
        \item Assume first that $M$ lies in the hereditary component of $J(X)$. Note that there are no maps from $X$ to any module in the hereditary component of $J(X)$, because the top of $X$ does not appear in any of the modules $\rad^1 X, \rad^2 X,\dots$ which by \cref{lem:projectives_in_jasso_reduction_c_n} are the relative projectives in the hereditary component of $J(X)$. Thus by \cref{lem:m_notinjasso_but_generated_by_jasso_characterization} all modules generated by $M$ in $\mods C_n$ are indeed modules in $J(X)$, so $\Gen M \cap J(X) = \Gen M$ and therefore $\Ext^1_{C_n}(N,\Gen_{C_n} M \cap J(X)) = \Ext^1_{C_n}(N,\Gen_{C_n} M)$ which by \cref{lem:ext_characterization_of_maps_to_tau} gives us our wanted conclusion.

         \item Assume now that $M$ lies in the cyclic component and $N$ in the hereditary component of $J(X)$. If $\tau N \in J(X)$ then $\tau N$ is also in the hereditary component: Observe that since $J(X)$ is closed under extensions, $\Ext^1_{C_n}(N,Y) = 0$ for any $Y$ in the cyclic component of $J(X)$. Now, as $\Ext_{C_N}^1(N,\tau N) \neq 0$ we get that $\tau N$ cannot lie in the cyclic component of $J(X)$. 
        We therefore in this case have $\Hom(M,\tau N)  = 0$ since $M$ and $\tau N$ again lie in different components of $J(X)$. Assume now that $\tau N \notin J(X)$.

        If $N$ is a quotient of $\rad^i X$ then $\tau N$ is a quotient of $\rad^i \tau X$ by  \cref{prop:nakayama_cn_properties} (3). Note that every proper quotient of $\rad^i \tau X$ lies in $J(X)$: If $\rad^i \tau X = M^t_l$ then $M^t_{l-1} = \rad^{i+1}X$. From $\tau N \notin J(X)$, we then get that $N = \rad^i X$.
        
        Thus $\tau N$ has an inclusion into $\tau X$. But then any nonzero map  in $\Hom(M,\tau N)$ would induce a nonzero map $M \to \tau X$, which we cannot have since $M \in J(X)$. Therefore, we also have $\Hom(M,\tau N) = 0$ if $\tau N \notin J(X)$.
        
        \item If we are not in the cases dealt with above, $M$ and $N$ both lie in the cyclic component of $J(X)$. Assume that $\Ext^1_{C_n}(N,\Gen M) \neq 0$. Let \[0 \to M' \xhookrightarrow{f} E \overset{g}\twoheadrightarrow N \to 0\] be a non-split short exact sequence with $M' \in \Gen M$. We will construct a non-split short exact sequence showing that $\Ext^1_{C_n}(N,\Gen_\Lambda M \cap J(X)) \neq 0$.
        
        By \cref{prop:nakayama_cn_properties}(2), there is an indecomposable direct summand $U$ of $E$ such that the co-restriction of $f$ to $U$ is an injection and an indecomposable direct summand $V$ of $E$ such that the restriction of $g$ to $V$ is an epimorphism. As the sequence is non-split, we have $\ell(U) >\ell(M')$ and $\ell(V) > \ell(N)$. Thus \[\ell(U)+\ell(V) > \ell(M')+\ell(N) = \ell(E)\] so $U$ and $V$ cannot be distinct indecomposable direct summands of $E$, and so we must have $U = V$.

        By \cref{lem:projectives_in_jasso_reduction_c_n}, the projectives in the cyclic component of $J(X)$ are projective in $\mod C_n$. The projective cover of $N$ lies in the cyclic component of $J(X)$ and must coincide with the projective cover of $U$, so $U$ is generated by some module in $J(X)$. 
        
        Let $U'$ be the module of minimum length in $J(X)$ generating $U$. If $U = M^t_l$, then $U' = M^t_{l+i}$ for some integer $i$. Let now $M''$ be the indecomposable module with the same top as $M'$ but of length $i$ more, that is $M'' = M^{\text{top}(M')}_{\ell(M')+i}$. We obtain a short exact, non-split sequence  
        \[0 \to M'' \xhookrightarrow{f'} U' \oplus E/U \overset{g'}\twoheadrightarrow N \to 0\] where $f'$ and $g'$ are the unique maps that are non-zero on each (co)-restriction. To complete the proof, it is sufficient to show that $M'' \in \Gen M \cap J(X)$. By \cref{lem:minimal_length_generator_in_cyclic_component} (setting $V' = M''$ and $V = M')$, we get that $M''$ is the module of least length in the cyclic component of $J(X)$ generating $M'$, so $M''$ is   generated by $M$.
        \end{enumerate}
\end{proof}
    
From the above lemma it follows that any module in $J(X) \subseteq \mods C_n$ is relatively $\tau$-rigid in $J(X)$ if and only if it is $\tau$-rigid in $\mods C_n$.

\begin{lemma}
   Let $(B,C)$ be a $\tau$-exceptional sequence in $J(X)$. Then $(B,C)$ is also a $\tau$-exceptional sequence in $\mods C_n$ and $\Psi^{-1}(B,C) = {\Psi^{-1}_{J(X)}}(B,C)$. 
\end{lemma}

\begin{proof}
    Let ${\Psi^{-1}_{J(X)}}(B,C) = B' \oplus C$. Since $B' \oplus C$ is $\tau_{J(X)}$-rigid, it is also $\tau$-rigid by \cref{lem:tau_rigid_in_J(x)_iff_tau_rigid}. Thus we compute $\Psi(B' \oplus C) = (f_C(B'),C) = {\Psi_{J(X)}}(B' \oplus C)$, where the last equality follows from noting that $J(X)$ is a full subcategory.
\end{proof}

We now prove the final technical result necessary to establish the main result of the paper.

\begin{proposition}\label{prop:property_2_cn}
    Property M2 in \cref{prop:properties_mutation_braid_group} holds for any algebra of the form $C_n$: Given an indecomposable $C_n$-module $X$ and a $\tau_{J(X)}$-exceptional sequence $(B',C) = \Psi(B \oplus C)$, we have \[\varphi^{J(X)}(B',C) = \varphi(B',C).\]
\end{proposition}

\begin{proof}

    We give a case-by-case proof depending on how $B$ and $C$ as modules in $J(X)$ embed into $\mods C_n$. Largely relying of the formulas in \cref{thm:bkt_tf_formula}, we show in each situation that $\overline{\varphi}^{J(X)}(B\oplus C) = \overline{\varphi}(B\oplus C)$, implying that $\varphi^{J(X)}(B',C) = \varphi(B',C)$.

    \begin{enumerate}
        \item Assume first that $B$ and $C$ lie in different components of $J(X)$. Then $\overline{\varphi}^{J(X)}(B \oplus C) = C \oplus B$ by \cref{lem:disconnect_mutation_swapping}. We now compute the mutation in $\mods C_n$. 

        If $C$ is projective in $\mods C_n$, we are in case TF-1a and so $\overline{\varphi}(B \oplus C) = C \oplus B$ as wanted. Assume that $C$ is not projective in $\mods C_n$. We wish to show that we are in case TF-3, namely that $C$ is neither in the Bongartz nor co-Bongartz complement of $B$. Since $C$ lies in a different component of $J(X)$, we have $\Hom_{C_n}(C,B) = 0 = \Hom_{C_n}(B,C)$. Note that non-projective summands of the Bongartz complement of $B$ have maps into $B$, and summands of the co-Bongartz complement of $B$ have maps coming from $B$. Therefore, if $C$ is not projective we are in case TF-3 and so $\overline{\varphi}(B\oplus C) = C\oplus B$ as wanted.
        
        \item Assume now that $B$ and $C$ both lie in the cyclic component of $J(X)$. Since $J(X)$ is full and since taking projective cover of a module (which appears in case TF-2b) in the cyclic component of $J(X)$ coincides with taking projective cover of the module in $\mods C_n$, it is sufficient to show that the TF-ordered module $B \oplus C$ lies in the same case (TF1-4) in $J(X)$ and in $\mods C_n$.

        \begin{enumerate}
            \item[(TF-1)]  Since $C$ lies in the cyclic component of $J(X)$, it is $C$ is projective in $J(X)$ if and only if it is projective in $\mods C_n$, so we are in TF-1a (b)  in $J(X)$ if and only if we are in TF-1a (b) in $\mods C_n$.
            \item[(TF-2)] If $C$ is generated by $B$ then we are in case TF-2 in both $J(X)$ and in $\mods C_n$. Since being projective in the cyclic component of $J(X)$ implies being projective in $\mods C_n$, we are in case TF2-a (b) in $J(X)$ exactly when we are in case TF-2a (b) in $\mods C_n$.
            \item[(TF-4)] Notice that the non-projective summands of the Bongartz complement of $B$ in $J(X)$ are the modules of the form $\rad^i B$ who also lie in $J(X)$. Thus we are in TF-4 in $J(X)$ exactly when we are in TF-4 in $\mods C_n$.
            \item[(TF-3)] If we are not in any of the above cases, we must be in case TF-3, both in $J(X)$ and in $\mods C_n$.
        \end{enumerate}

        \item Assume that $B$ and $C$ both lie in the hereditary component of $J(X)$. Note that the projectives in this subcategory are exactly the modules of the form $\rad^i X$ for $i \geq 1$. In particular, none of these are projective in $\mods C_n$, which we must take into consideration in TF-2b. We must also keep in mind that taking (co-)Bongartz complements differ in $J(X)$ and in $\mods C_n$. 

        For each case in $J(X)$, we now identify the corresponding case in $\mods C_n$ and compare the computations of $\overline{\varphi}$.

        \begin{enumerate}
            \item[(TF-1a)]  If we are in Case TF-1a in $J(X)$, then $C$ is not of the form $\rad^i B$ and so we must be in case TF-3 in $\mods C_n$. Computing $\overline{\varphi}$ is given by the same formula for Case TF-1a and TF-3.

            \item[(TF-1b)]  If we are in Case TF-1b in $J(X)$, then by $\tau$-rigidity and TF-ordering of $B \oplus C$ we must have $C \cong \rad^i B$ and so we are in case TF-4 in $\mods C_n$. Computing $\overline{\varphi}$ is given by the same formula for Case TF-1b and TF-4.

            \item[(TF-2a)]  If we are in case TF-2a in $J(X)$, then $C \cong B/\rad^i B$ for some $i \geq 1$ and so we are also in case TF-2a in $\mods C_n$.

            \item[(TF-2b)] If we are in case TF-2b in $J(X)$ then note first that $P_{J(X)}(\rad^k B) = \rad^k B$ since we are in a hereditary subcategory of $J(X)$, so $\overline{\varphi}^{J(X)}(B \oplus C) = \rad^k B \oplus B$. In $\mods C_n$, we would be in case TF-2a, since no indecomposable module in the hereditary component of $J(X)$ is projective in $\mods C_n$. We thus get $\overline{\varphi}(B \oplus C) = \rad^k B \oplus B$ as wanted.

            \item[(TF-3)] Observe that the modules $\rad^i B$ ($B/\rad^i B$) are exactly the summands of the (co) Bongartz complement of $B$ taken in $\mods C_n$ that also lie in the hereditary component of $J(X)$, and that these summands also appear in the (co) Bongartz complement of $B$ in $J(X)$.  Thus, if we are in TF-3 in $J(X)$, then we are in TF-3 in $\mods C_n$ too.
            
            \item[(TF-4)] If we are in case TF-4 in $J(X)$ then $B$ is not projective in $J(X)$ and so $C = \rad^iB$ for some $i \geq 1$ and so we are in case TF-4 also in $\mods C_n$.
        \end{enumerate}

    \end{enumerate}
\end{proof}

\begin{proposition}\label{prop:property_2_N}
    Property M2 in \cref{prop:properties_mutation_braid_group} holds for any algebra in $\mathcal{N}$. 
\end{proposition}

\begin{proof}
    Let $\Lambda$ be an algebra in $\mathcal{N}$ and let $X$ be indecomposable $\tau$-rigid over $\Lambda$.

    We can only have $\varphi^{J(X)}(B,C) \neq \varphi(B,C)$ if $B$, $C$ and $X$ all lie in the same component of $\mods \Lambda$: If $B$ and $C$ lie in different components, then $\varphi(B,C) = (C,B) = \varphi^{J(X)}(B,C)$. If $B$ and $C$ lie in the same component but $X$ lies in a different component then considering the reduction $J(X)$ does not influence the mutation so $\varphi(B,C) = \varphi^{J(X)}(B,C)$.
    
    Assume therefore that $B$, $C$ and $X$ lie in the same component $\mathcal{C}$ of $\mods \Lambda$. Either we have $\mathcal{C} = \mods C_i$ or $\mathcal{C} = \mods A_i$ for some integer $i$. In the first case, \cref{prop:property_2_cn} shows that $\varphi^{J(X)}(B,C) = \varphi(B,C)$. In the second case, we are in a hereditary algebra where the property always holds by \cref{prop:hereditary_properties_mutation_braid_group}.
\end{proof}

\begin{theorem}\label{thm:braid_group_relations_satisfied_on_cn}
    For any algebra in $\mathcal{N}$, in particular for any cyclic Nakayama algebra of the form $C_n =\overrightarrow{\Delta}_n/r^n$, mutation of complete $\tau$-exceptional sequences respects the braid group relations on $n$ strands.
\end{theorem}

\begin{proof}
    Since $\mathcal{N}$ is closed under Jasso reductions, \cref{prop:property_1_n} and \cref{prop:property_2_N} allows us to apply \cref{prop:properties_mutation_braid_group} showing that the braid group relations are satisfied. 
\end{proof}

\section{Braid group relations for Nakayama algebras having $C_n$ as a quotient}\label{sec:longerprojs}
In this section we extend \cref{thm:braid_group_relations_satisfied_on_cn} to cyclic Nakayama algebras, all of whose indecomposable projective modules are of length at least $n$ by showing the following result. 

\begin{theorem}\label{thm:longerlengths}
    Let $\Lambda_1$ and $\Lambda_2$ be two basic and connected Nakayama algebras with $|\Lambda_1|=|\Lambda_2|$ such that $\ell(P) \geq |\Lambda_i| \leq \ell(Q)$ for all indecomposable projective $\Lambda_1$-modules $P$ and indecomposable projective $\Lambda_2$-modules $Q$. Then there is a bijection $\xi$ between $\tau$-exceptional sequences in $\mods \Lambda_1$ and $\tau$-exceptional sequences in $\mods \Lambda_2$. Moreover, this bijection is compatible with the mutation of complete $\tau$-exceptional sequences. More precisely, we have
    \[ \xi(\varphi_i(\mathcal{M})) = \varphi_i(\xi(\mathcal{M}))\]
    for all complete $\tau$-exceptional sequence $\mathcal{M}$ in $\mods \Lambda_1$ and all $1 \leq i < |\Lambda_1|$. Consequently, the mutation of $\tau$-exceptional sequences in $\mods \Lambda_1$ and in $\mods \Lambda_2$ satisfies the braid group relations.
\end{theorem}

Recall that any basic Nakayama algebra $\Lambda$ with $n$ isomorphism classes of simple modules is determined up to isomorphism by the lengths of its indecomposable projective modules, known as its Kupisch series. 

\begin{setting}\label{set:sec5}
    Let $\Lambda_1$ and $\Lambda_2$ be two basic and connected Nakayama algebras with $n \coloneqq |\Lambda_1|=|\Lambda_2|$. Assume that they are defined via their Kupisch series as follows: 
    \[ \Lambda_1 = \mathcal{N}(m_0, m_1, \dots, m_{n-1}), \quad \text{and} \quad \Lambda_2 = \mathcal{N}(m_0', m_1', \dots, m_{n-1}'), \]
    that is, both have $n \geq 1$ isomorphism classes of simple modules and their $i$-th indecomposable projective modules have length $m_i$ and $m_i'$ respectively. Assume furthermore that $m_i = m_i' \geq n$ for $i \geq 1$ and $m_0' = m_0+1 \geq n+1$. 
\end{setting}

In other words, all indecomposable projective modules of the two algebras in \cref{set:sec5} have length at least $n$, and differ only in the length of the projective at vertex 0. One easily checks that in order to prove \cref{thm:longerlengths} it is sufficient to consider two algebras satisfying \cref{set:sec5}. Such algebras satisfy:
\[ \ind(\mods \Lambda_2) = \ind(\mods \Lambda_1) \cup \{ P_{\Lambda_2}(0)\}, \quad\text{and} \quad \ind (\taur \Lambda_2) = (\ind (\taur \Lambda_1) \setminus \{ P_{\Lambda_1}(0)\}) \cup \{ P_{\Lambda_2}(0)\}. \]

In this case, the module $P_{\Lambda_2}(0)$ is projective-injective in $\mods \Lambda_2$ and this approach is similar to Drozd--Kirichenko rejection \cite{DrozdKirichenko} as applied to Nakayama algebras in \cite[Sec. 3]{Adachi2016}. We collect some preliminary results.

\begin{lemma}\label{lem:easylemma}
    Let $\Lambda_1$ and $\Lambda$ be as in \cref{set:sec5}. Let $X \neq P_{\Lambda_1}(0)$ be an indecomposable $\tau$-rigid module in $\mods \Lambda_1$, which we may also view as an indecomposable $\tau$-rigid module in $\mods \Lambda_2$. Then:
    \begin{enumerate}
        \item $P_{\Lambda_1}(0) \oplus X$ is $\tau$-rigid in $\mods \Lambda_1$ if and only if $P_{\Lambda_2}(0) \oplus X$ is $\tau$-rigid in $\mods \Lambda_2$;
        \item $X \oplus P_{\Lambda_1}(0)$ is TF-ordered in $\mods \Lambda_1$ if and only if $X \oplus P_{\Lambda_2}(0)$ is TF-ordered in $\mods \Lambda_2$;
        \item $\Hom_{\Lambda_1}(P_{\Lambda_1}(0), X)= 0$ if and only if $\Hom_{\Lambda_2}(P_{\Lambda_2}(0), X)= 0$;
        \item $f_{P_{\Lambda_1}(0)}^{\Lambda_1} X= f_{P_{\Lambda_2}(0)}^{\Lambda_2} X$;
        \item $X \in \Gen^{\Lambda_1} P_{\Lambda_1}(0)$ if and only if $X \in \Gen^{\Lambda_2} P_{\Lambda_2}(0)$ and in this case $X \cong  P_{\Lambda_1}(0)/ \rad^j (P_{\Lambda_1}(0))$ if and only if $X \cong  P_{\Lambda_2}(0)/ \rad^j (P_{\Lambda_2}(0))$;
    \end{enumerate}
\end{lemma}
\begin{proof}
    \begin{enumerate}
        \item The module $P_{\Lambda_i}(0) \oplus X$ is $\tau$-rigid in $\mods \Lambda_i$ if and only if $X$ is $\tau$-rigid in $\mods \Lambda_i$ and $\Hom_{\Lambda_i}(P_{\Lambda_i}(0), \tau_{\Lambda_i} X) = 0$ for $i=1,2$. Since $\tau_{\Lambda_1} X = \tau_{\Lambda_2} X$ as quiver representations, they either both contain $0$ as a composition factor (equivalently are supported at vertex 0) or not. By \cite[Lem. III.2.11, Cor. III.3.6]{assem_skowronski_simson_2006}, the result follows.
        \item $X$ is generated by $P_{\Lambda_i}(0)$ in $\mods \Lambda_i$ if and only if $\topp(X) \cong S(0)$, by the uniseriality of indecomposable modules. Because $\topp(X)$ is independent of the ambient category, the result follows. 
        \item Either $X$ contains $0$ as a composition factor (equivalently is supported at vertex 0) or not, regardless of the ambient category. By \cite[Lem. III.2.11, Cor. III.3.6]{assem_skowronski_simson_2006}, the result follows.
        \item Since $P_{\Lambda_2}(0)$ is projective, the image of any morphism $f: P_{\Lambda_2}(0) \to X$ is a proper quotient of $P_{\Lambda_2}(0)$. By uniseriality, it is also a quotient module of $(P_{\Lambda_2}(0))/\soc(P_{\Lambda_2}(0)) = P_{\Lambda_1}(0)$. The result follows.
        \item The first part follows from (2). In this case, it follows from (the proof of) \cite[Thm. V.3.5]{assem_skowronski_simson_2006} that $j= \ell(X)$ gives the desired equality. 
    \end{enumerate}
    This concludes the proof.
\end{proof}

There exists a bijection between complete $\tau$-exceptional sequences of any two $\tau$-tilting finite algebras $\Gamma_1$ and $\Gamma_2$ such that $\stt \Gamma_1  \cong \stt \Gamma_2$ is an isomorphism of posets by \cite[Thm. D]{BarnardHanson2022exc}. By \cite[Thm. 3.11]{Adachi2016}, this holds for Nakayama algebras $\Lambda_1$ and $\Lambda_2$ as in the setting of \cref{lem:easylemma}. Thus, \cite[Thm. 5.1]{mendozatreffinger_stratifyingsystems} implies a slightly weaker version of the following lemma. 
\begin{lemma}\label{lem:bijTFtt}
    Let $\Lambda_1$ and $\Lambda_2$ be as in \cref{set:sec5}. Then there exists a bijection 
    \[ \Xi: \{ \tfo{\Lambda_1} \} \to \{ \tfo{\Lambda_2} \} \]
    given by replacing every occurrence of $P_{\Lambda_1}(0)$ by $P_{\Lambda_2}(0)$ and leaving the other terms unchanged.
\end{lemma}
\begin{proof}
    Since $\mods \Lambda_1 \subseteq \mods \Lambda_2$, every $\Lambda_1$-module can be viewed as a $\Lambda_2$-module. The uniseriality of indecomposable $\Lambda_i$-modules implies that it is sufficient to check whether a $\Lambda_i$-module $X$ generates an indecomposable $\Lambda_i$ module $Y$ by checking whether any indecomposable direct summand of $X$ generates $Y$ for $i=1,2$. Therefore \cref{lem:easylemma}(1) and \cref{lem:easylemma}(2) imply the desired result. 
\end{proof}

In combination with \cref{thm:MendozaTreffinger}, \cref{lem:bijTFtt} gives the desired bijection for \cref{thm:longerlengths}. We need the following result on $\tau$-tilting reduction of Nakayama algebras since the mutation of TF-ordered $\tau$-rigid modules with more than two indecomposable direct summands happens in a $\tau$-perpendicular subcategory.

\begin{lemma}\label{lem:longreduction}
    Let $\Lambda = \mathcal{N}(m_0, \dots, m_{n-1})$ be a basic Nakayama algebra such that $m_i \geq n$ for all $0 \leq i \leq n-1$ and some $n \geq 1$. Let $X$ be an indecomposable nonprojective $\Lambda$-module such that $P(0) \oplus X$ is $\tau$-rigid. Consider the $\tau$-tilting reduction $J(X)$ of $\mods \Lambda$ with respect to $X$. The (relative) projective object $f_X(P(0))$ in $J(X)$ lies in a connected component of $J(X)$ which is isomorphic to $\mods \Gamma$ for a Nakayama algebra $\Gamma$ all of whose indecomposable projectives are of length greater than or equal to $|\Gamma|$. 
\end{lemma}
\begin{proof}
    By \cite{jassoreduction}, there is an isomorphism of categories $J(X) \cong \End(\bcomp(X))/[X]$ and by \cite[Lem. 4.7]{bkt} the subcategory $J(X)$ is Morita equivalent to a (possibily disconnected) Nakayama algebra. From the description of  $\bcomp(X)$ in \cite[Prop. 4.9]{bkt}, we know that the (relative) projective modules of $J(X)$ are 
    \[ f_X(P(t)), \dots, f_X(P(t+n-\ell(X)-1)_n), \quad \text{and}\quad \rad^1 X, \dots, \rad^{\ell(X)-1} X, \] 
    where $t$ is such that $\topp(X) \cong S(t)$ and $(t+n-\ell(X)-1)_n$ means that the sum is taken modulo $n$. By \cite[Prop. 6.5]{bkt}, these two classes of projectives lie in disconnected components of $J(X)$. Since $P(0) \oplus X$ is $\tau$-rigid it follows that $f_X(P(0))$ is an element of  
    \[ \Pcal = \{f_X (P(j)): j \in t, \dots, (t+n-\ell(X)-1)_n\}.\]
    It remains to show that each element of $\{P(t), \dots, P(t+n-\ell(X)-1)\}$ has a nonzero morphism, which does not factor through $X$, to all other elements. This implies that $\Pcal$ generates a connected component of $J(X)$ with $|n- \ell(X)|$ isomorphism classes of simples and moreover that the indecomposable projective objects $\Pcal$ of $J(X)$ all have length at least $|n- \ell(X)|$, concluding the proof. Since the lengths of all projective $\Lambda$-modules are greater than or equal to $|\Lambda|$, there are nonzero left $\add(P(j))$-approximations $f_{ij}: P(i) \to P(j)$ for any $i,j \in \{ t, \dots, (t+n-\ell(X)-1)_n\}$. Note that $\Hom(P(i),X)\neq 0$ if and only if $i = t$ for those values of $i$, see \cite[Prop. 4.6(10)]{bkt}. Consequently $f_{ij}$ does not factor through $X$ for $i\neq t$. For $i = t$, the left $\add(P(t+a))$-approximation $f_a: P(t) \to P(t+a)$, with $a \in \{1, \dots, n-\ell(X)-1\}$, factors through $\rad^a P(t+a)$. We have
    \begin{align*} 
    \ell(\rad^a P(t+a)) &= \ell(P(t+a)) - a \\
    & \geq n -a &\text{ (since $\ell(P(t+a))\geq n$)}\\
    & \geq n- (n-\ell(X)-1) &\text{ (since $a \leq n-\ell(X)-1$)} \\
    & \geq \ell(X)+1.
    \end{align*}
    Therefore the morphism $f_a: P(t) \to P(t+a)$ does not factor through $X$. It follows that $f_X(P(0))$ lies in a connected component of $J(X)$ which is isomorphic to a cyclic Nakayama algebra $\Gamma$ all of whose indecomposable projectives have length at least $|\Gamma|$. 
\end{proof}

Two modules lying in different connected components of a disconnected algebra are always $\Hom$-orthogonal and $\Ext$-orthogonal to each other. Moreover, if they are both $\tau$-rigid, then their direct sum is $\tau$-rigid. Using this idea, \cref{lem:longreduction} allows us to argue that we remain in \cref{set:sec5} when performing $\tau$-tilting reduction in later proofs. For example, we apply it to obtain the following extension of \cref{lem:bijTFtt}.

\begin{lemma}\label{lem:reducedbijection}
    Let $\Lambda_1$ and $\Lambda_2$ be as in \cref{set:sec5}. Let $X \in \mods \Lambda_1$ be indecomposable $\tau$-rigid. Then the bijection $\Xi$ of \cref{lem:bijTFtt} induces a bijection
    \[ \Xi^{X}: \{ \tfo{(J^{\Lambda_1}(X))}\} \to \{ \tfo{(J^{\Lambda_2}(\Xi(X)))}\}\]
    given by replacing every occurrence of $f_X^{\Lambda_1} (P_{\Lambda_1}(0))$ by $f_{X}^{\Lambda_2}(P_{\Lambda_2}(0))$.
\end{lemma}
\begin{proof}
    If $X \cong P_{\Lambda_1}(0)$, it follows immediately from \cref{lem:easylemma}(3) that $J^{\Lambda_1}(X) = J^{\Lambda_2}(\Xi(X))$ and the bijection is trivial. 

    If $X$ is isomorphic to another indecomposable projective, then $\Hom_{\Lambda_2}(X,P_{\Lambda_2}(0)) \neq 0$ since $\ell(P_{\Lambda_2}(0)) \geq n$. It follows that $J^{\Lambda_1}(X) = J^{\Lambda_2}(X) = J^{\Lambda_2}(\Xi(X))$ because $P_{\Lambda_2}(0) \not \in J^{\Lambda_2}(X)$ and the bijection is again trivial. 
    
    If $X$ is not projective and $P_{\Lambda_1}(0) \oplus X$ is not $\tau$-rigid, then $P_{\Lambda_2}(0) \oplus X$ is not $\tau$-rigid by \cref{lem:easylemma}(1). This means that $\Hom(P_{\Lambda_2}(0), \tau X) \neq 0$ and consequently, $J^{\Lambda_1}(X) = J^{\Lambda_2}(X) = J^{\Lambda_2}(\Xi(X))$ because $P_{\Lambda_2}(0) \not \in J^{\Lambda_2}(X)$. The bijection is again trivial. If $X$ is not projective and $P_{\Lambda_2}(0) \oplus X$ is $\tau$-rigid, then we have 
    \begin{equation}\label{eq:jassosplit} J^{\Lambda_1}(X) \cong \mods \Gamma_1 \times \mods \Gamma' \text{ and } J^{\Lambda_2}(\Xi(X)) \cong \mods \Gamma_2 \times \mods \Gamma' \end{equation}
    where $\Gamma_1$ and $\Gamma_2$ satisfy \cref{set:sec5}, by \cref{lem:longreduction}. Moreover, by the same lemma $f_X^{\Lambda_1}(P_{\Lambda_1}(0))$ lies in the component isomorphic to $\mods \Gamma_1$, whereas $f_X^{\Lambda_2}(P_{\Lambda_2}(0))$ lies in the component isomorphic to $\mods \Gamma_2$. The result then follows from applying \cref{lem:bijTFtt} to $\Gamma_1$ and $\Gamma_2$ and combining it with the identity bijection of TF-ordered $\tau$-rigid $\Gamma'$-modules.
\end{proof}

Using \cite[Thm. 4.3]{BuanMarsh2018}, we can extend this bijection additively to a bijection
\[ \Xi^{X \oplus Y} : \{\tfo{(J^{\Lambda_1}(X \oplus Y))}\} \to \{\tfo{(J^{\Lambda_2}(\Xi(X \oplus Y)))}\} \]
for any TF-ordered $\tau$-rigid module $X \oplus Y$.

\begin{remark}
    Let $X$ be $\tau$-rigid. The functor $f_X(-)$ gives a bijection between indecomposable modules $Y \in \mods \Lambda$ such that $Y \oplus X$ is TF-ordered $\tau$-rigid and indecomposable $\tau$-rigid modules in $J(X)$ by \cite[Prop. 4.5]{BuanMarsh2018}. For an indecomposable $\tau$-rigid module $N \in J(X)$ we write $f_X^{-1}(N)$ for the inverse image of this bijection.  
\end{remark}

\begin{lemma}\label{lem:reductionbijcommutes}
    The bijection of \cref{lem:reducedbijection} satisfies
    \[ \Xi^X ( f_X^{\Lambda_1}(Z)) = f_{\Xi(X)}^{\Lambda_2}( \Xi(Z)) \]
    for every indecomposable $Z \in \mods \Lambda_1$ such that $Z \oplus X$ is TF-ordered $\tau$-rigid. Moreover, the bijection satisfies
    \[ \Xi ( (f_X^{\Lambda_1})^{-1}(Z')) = (f_{\Xi(X)}^{\Lambda_2})^{-1} ( \Xi^{X}(Z')) \]
    for every indecomposable $\tau$-rigid object $Z'$ in $J^{\Lambda_1}(X)$.
\end{lemma}
\begin{proof}
    If $Z \not \cong P_{\Lambda_1}(0)$, then trivially $\Xi(Z)=Z$ and $\Xi^X(f_X^{\Lambda_1}(Z)) = f_X^{\Lambda_1}(Z)$ hold. Therefore 
    \[ f_{\Xi(X)}^{\Lambda_2}(\Xi(Z)) = f_{\Xi(X)}^{\Lambda_2}(Z) = f_X^{\Lambda_1}(Z) = \Xi^X(f_X^{\Lambda_1}(Z)),\]
    where the second equality follows from \cref{lem:easylemma}(4).

    If $Z \cong P_{\Lambda_1}(0)$, then $X \not \cong P_{\Lambda_1}(0)$ so that $\Xi(X)=X$ and thus
    \[ f_{\Xi(X)}^{\Lambda_2}(\Xi(Z)) = f_X^{\Lambda_2}(P_{\Lambda_2}(0)) = \Xi^X(f_X^{\Lambda_1}(P_{\Lambda_1}(0))). \]

    We may assume that $Z' = f_X(Y)$ for some indecomposable $Y \in \mods \Lambda_1$ such that $Y \oplus X$ is TF-ordered $\tau$-rigid by \cite[Prop. 4.5]{tauexceptional_buanmarsh}. If $Y \not \cong P_{\Lambda_1}(0)$, then $\Xi(Y) = Y$ and $\Xi^X(Z') = Z'$ hold. Thus
    \begin{align*} (f_{\Xi(X)}^{\Lambda_2})^{-1}( \Xi^X(Z')) & = ( f_{\Xi(X)}^{\Lambda_2})^{-1}(Z') \\
    & = ( f_{\Xi(X)}^{\Lambda_2})^{-1}(f_X^{\Lambda_1}(Y)) \\
    &= (f_{\Xi(X)}^{\Lambda_2})^{-1} (f_{\Xi(X)}^{\Lambda_2}(Y))  & \text{(by \cref{lem:easylemma}(2))}\\
    &= Y \\
    & = \Xi(Y) \\
    &= \Xi((f_{X}^{\Lambda_1})^{-1}( f_X^{\Lambda_1}(Y))) \\
    &= \Xi((f_X^{\Lambda_1})^{-1}(Z'))\end{align*}
    If $Z' \cong f_X^{\Lambda_1} (P_{\Lambda_1}(0))$, then $X \not \cong P_{\Lambda_1}(0)$ so that $\Xi(X)=X$ and we have
    \[ (f_{\Xi(X)}^{\Lambda_2})^{-1} ( \Xi^X(Z')) = (f_X^{\Lambda_2})^{-1} ( f_X^{\Lambda_2}(P_{\Lambda_2}(0))) = P_{\Lambda_2}(0) = \Xi( (f_X^{\Lambda_1})^{-1} (f_X^{\Lambda_1}(P_{\Lambda_1}(0)))) = \Xi((f_X^{\Lambda_1})^{-1}(Z')). \]
    This concludes the proof.
\end{proof}

In the proof of the following lemma we rely on \cite[Thm. 4.3]{BuanMarsh2018} to extend \cref{lem:reductionbijcommutes} to the following setting.
\begin{lemma}\label{lem:reductionstep}
    Let $\Lambda_1$ and $\Lambda_2$ be as in \cref{set:sec5} and let $Z \oplus X \oplus Y$ be TF-ordered $\tau$-rigid in $\mods \Lambda_1$. Then
    \[ \Xi( (f_{\overline{\varphi}^{\Lambda_1}(X \oplus Y)}^{\Lambda_1})^{-1} ( f_{X \oplus Y}^{\Lambda_1}(Z))) = (f_{\Xi(\overline{\varphi}^{\Lambda_1}(X \oplus Y))}^{\Lambda_2})^{-1} ( f_{\Xi(X \oplus Y)}^{\Lambda_2}(\Xi(Z))).\]
\end{lemma}
\begin{proof}
    The result follows from the sequence of equalities
    \begin{align*}
        (f_{\Xi(\overline{\varphi}^{\Lambda_1}(X \oplus Y))}^{\Lambda_2})^{-1} ( f_{\Xi(X \oplus Y)}^{\Lambda_2}(\Xi(Z))) &= (f_{\Xi(\overline{\varphi}^{\Lambda_1}(X \oplus Y))}^{\Lambda_2})^{-1} ( \Xi^{X \oplus Y}(f_{X \oplus Y}^{\Lambda_1}(Z))) &\text{(by \cref{lem:reductionbijcommutes})} \\
        &= (f_{\Xi(\overline{\varphi}^{\Lambda_1}(X \oplus Y))}^{\Lambda_2})^{-1} ( \Xi^{\overline{\varphi}^{\Lambda_1}(X \oplus Y)}(f_{X \oplus Y}^{\Lambda_1}(Z))), \\
        &= \Xi((f_{\overline{\varphi}^{\Lambda_1}(X \oplus Y)}^{\Lambda_1})^{-1} ( f_{X \oplus Y}^{\Lambda_1}(Z))) &\text{(by \cref{lem:reductionbijcommutes})}
    \end{align*}
    where the second equality follows from $\Xi^{X \oplus Y} = \Xi^{\overline{\varphi}^{\Lambda_1}(X \oplus Y)}$. These bijections are equivalent because $J^{\Lambda_1}(X \oplus Y) = J^{\Lambda_1}(\overline{\varphi}^{\Lambda_1}(X \oplus Y))$ by \cite[Def.-Prop. 4.3 and 4.4]{BHM2024} and \cite[Thm. 6.4]{buan2023perpendicular}.
\end{proof}

The following two lemmas illustrate that the mutation of TF-ordered modules with two indecomposable direct summands (and equivalently of $\tau$-exceptional pairs) coincide for two Nakayama algebras $\Lambda_1$ and $\Lambda_2$ as in \cref{set:sec5}. In other words, up to replacing every occurrence of $P_{\Lambda_1}(0)$ by $P_{\Lambda_2}(0)$, the mutation of $\tau$-exceptional pairs coincide in $\mods \Lambda_1$ and $\mods \Lambda_2$. 

\begin{lemma}\label{lem:mutcase1}
    Let $\Lambda_1$, $\Lambda_2$ be as in \cref{set:sec5}. Let $X \neq P_{\Lambda_1}(0)$ be an indecomposable $\tau$-rigid module in $\mods \Lambda_1$ such that $P_{\Lambda_1}(0) \oplus X$ is $\tau$-rigid. As TF-ordered $\tau$-rigid modules, the module $P_{\Lambda_1}(0) \oplus X$ falls into Case TF-\textnormal{?} in $\mods \Lambda_1$ if and only if the module $P_{\Lambda_2}(0) \oplus X$ falls into Case TF-\textnormal{?} in $\mods \Lambda_2$, where $? \in \{1b, 2b, 3\}$ and they both do not fall into any other cases. Moreover, in Case TF-1b, we have
    \[ \overline{\varphi}^{\Lambda_1}( P_{\Lambda_1}(0) \oplus X) = P_{\Lambda_1}(0) \oplus Y, \quad \text{and} \quad \overline{\varphi}^{\Lambda_2}(P_{\Lambda_2}(0) \oplus X) = P_{\Lambda_2}(0) \oplus Y, \]
    where $Y = f_X^{\Lambda_1}(P_{\Lambda_1}(0)) = f_X^{\Lambda_2}(P_{\Lambda_2}(0))$ by \cref{lem:easylemma}(5). In Case TF-2b, we have
    \[ \overline{\varphi}^{\Lambda_1}( P_{\Lambda_1}(0) \oplus X) = Y' \oplus P_{\Lambda_1}(0) , \quad \text{and} \quad \overline{\varphi}^{\Lambda_2}(P_{\Lambda_2}(0) \oplus X) = Y' \oplus P_{\Lambda_2}(0), \]
    where $X \cong (P_{\Lambda_i}(0))/\rad^j(P_{\Lambda_i}(0))$ for $i=1,2$ by \cref{lem:easylemma}(5) and $Y' = P_{\Lambda_i}(\rad^j(P_{\Lambda_i}(0)))$ is the projective cover of $\rad^j(P_{\Lambda_i}(0))$ in $\mods \Lambda_i$ which coincide because $\ell(X) < n$ implies that $Y' \not \cong P_{\Lambda_i}(0)$. Finally, in case TF-3, we have
    \[ \overline{\varphi}^{\Lambda_1}( P_{\Lambda_1}(0) \oplus X) = X \oplus P_{\Lambda_1}(0), \quad \text{and} \quad \overline{\varphi}^{\Lambda_2}( P_{\Lambda_2}(0) \oplus X) = X \oplus P_{\Lambda_2}(0). \]
\end{lemma}
\begin{proof}
    First of all note that $P_{\Lambda_1}(0) \oplus X$ is a $\tau$-rigid module in $\mods \Lambda_1$ if and only if $P_{\Lambda_2}(0) \oplus X$ is a $\tau$-rigid module in $\mods \Lambda_2$ by \cref{lem:easylemma}(1). They are then immediately TF-ordered since $P_{\Lambda_i}(0)$ is projective in $\mods \Lambda_i$.  Furthermore, note that $X$ is projective in $\mods \Lambda_1$ if and only if it projective in $\mods \Lambda_2$. If $X$ is projective, then $\ell(X) \geq n$ and we have $\Hom_{\Lambda_i}(X, P_{\Lambda_i}(0)) \neq 0$ for $i =1,2$. Consequently, $P_{\Lambda_i}(0) \oplus X$ falls into Case TF-1b and never falls into Case TF-1a in $\mods \Lambda_i$. 
    
    If $X$ is not projective in $\mods \Lambda_i$, then since $P_{\Lambda_i}(0)$ is projective in $\mods \Lambda_i$ the TF-ordered module $P_{\Lambda_i}(0) \oplus X$ may not fall into Case TF-2a. Similarly since $P_{\Lambda_i}(0)$ is projective, there is an isomorphism $\bcomp(P_{\Lambda_i}(0)) \cong {}_{\Lambda_i}\Lambda_i$. Hence, the Bongartz completion has no nonprojective direct summands, whence $P_{\Lambda_i}(0) \oplus X$ does not fall into Case TF-4. 

    By \cref{lem:easylemma}(5) if either $P_{\Lambda_1}(0) \oplus X$ or $P_{\Lambda_2}(0) \oplus X$ falls into Case TF-2b, so does the other. Moreover, since $\ell(P_{\Lambda_1}(0)) \geq n$, we have that $\ccomp(P_{\Lambda_1}(0)) = \ccomp(P_{\Lambda_2}(0))$ by the description of \cite[Prop. 4.11]{bkt}. We thus have $X \in \add(\ccomp(P_{\Lambda_1}(0)))$ if and only if $X \in \add(\ccomp(P_{\Lambda_2}(0)))$, hence $P_{\Lambda_1}(0) \oplus X$ falls into Case TF-3 in $\mods \Lambda_1$ if and only if $P_{\Lambda_2}(0) \oplus X$ falls into Case TF-3 in $\mods \Lambda_2$. The result follows from \cref{thm:bkt_tf_formula}, using \cref{lem:easylemma} as indicated in the statement of this claim.
\end{proof}

\begin{lemma}\label{lem:mutcase2}
    Let $\Lambda_1$, $\Lambda_2$ and $X$ be as in the setting of \cref{lem:mutcase1}. If $X \oplus P_{\Lambda_1}(0)$ is TF-ordered $\tau$-rigid in $\mods \Lambda_1$, it falls into case TF-1x in $\mods \Lambda_1$ if and only if $X \oplus P_{\Lambda_2}(0)$ falls into Case TF-1x in $\mods \Lambda_2$ where $x \in \{a,b\}$ and they both do not fall into any other cases. Moreover, in Case TF-1a, we have
    \[ \overline{\varphi}^{\Lambda_1}(X \oplus P_{\Lambda_1}(0)) =  P_{\Lambda_1}(0) \oplus X, \quad \text{and} \quad \overline{\varphi}^{\Lambda_2}(X \oplus P_{\Lambda_2}(0)) =  P_{\Lambda_2}(0) \oplus X.\]
    In Case TF-1b, we have
    \[ \overline{\varphi}^{\Lambda_1}(X \oplus P_{\Lambda_1}(0)) =  X \oplus Y'', \quad \text{and} \quad \overline{\varphi}^{\Lambda_2}(X \oplus P_{\Lambda_2}(0)) =  X \oplus Y'',\]
    where $Y'' = f_{P_{\Lambda_1}(0)}^{\Lambda_1}(X) = f_{P_{\Lambda_2}(0)}^{\Lambda_2}(X)$ by \cref{lem:easylemma}(4).
\end{lemma}
\begin{proof}
   Note that by \cref{lem:easylemma}(1,2), $X \oplus P_{\Lambda_1}(0)$ is TF-ordered $\tau$-rigid in $\mods \Lambda_1$ if and only if $X \oplus P_{\Lambda_2}(0)$ is TF-ordered $\tau$-rigid in $\mods \Lambda_2$. Since $P_{\Lambda_i}(0)$ is projective in $\mods \Lambda_i$, the module $X \oplus P_{\Lambda_i}(0)$ falls into Case TF-1 in $\mods \Lambda_i$ for $i=1,2$. They both fall into the same subcase because of \cref{lem:easylemma}(3) and the result follows from \cref{thm:bkt_tf_formula}. 
\end{proof}

As a consequence we describe the interaction of the mutation of TF-ordered $\tau$-rigid modules with two indecomposable direct summands with the bijection $\Xi$.

\begin{proposition}\label{prop:mutationstep}
    Let $\Lambda_1$ and $\Lambda_2$ be as in \cref{set:sec5} and let $X \oplus Y$ be a TF-ordered $\tau$-rigid module in $\mods \Lambda_1$. Then
    \[ \Xi(\overline{\varphi}^{\Lambda_1}(X \oplus Y)) = \overline{\varphi}^{\Lambda_2}(\Xi(X \oplus Y)).\]
\end{proposition}
\begin{proof}
    If $X \cong P_{\Lambda_1}(0)$ this is an immediate consequence of \cref{lem:mutcase1} and if $Y \cong P_{\Lambda_1}(0)$ it is an immediate consequence of \cref{lem:mutcase2}. Otherwise, the descriptions of $\bcomp(X)$ and $\ccomp(X)$ given in \cite[Prop. 4.9, Prop. 4.11]{bkt} yield that $X \oplus Y$ falls into case Case TF-? in $\mods \Lambda_1$ if and only if $\Xi(X \oplus Y) = X \oplus Y$ falls into case TF-?  in $\mods \Lambda_2$, where $? \in \{1a, 1b, 2a, 2b, 3 , 4\}$. If they fall into Case TF-2b, then $Y \cong X/\rad^j X$ for some $1 \leq j \leq \ell(X)-1$ regardless of the ambient category. Then, \cref{thm:bkt_tf_formula} gives
    \[ \overline{\varphi}^{\Lambda_1}(X \oplus Y) \cong P_{\Lambda_1}(\rad^j X) \oplus X \quad \text{ and } \quad \overline{\varphi}^{\Lambda_2}(X \oplus Y) \cong P_{\Lambda_2}(\rad^j X) \oplus X.\]
    It is clear that $P_{\Lambda_1}(\rad^j X) \cong P_{\Lambda_1}(0)$ if and only if $ P_{\Lambda_2}(\rad^j X) \cong P_{\Lambda_2}(0)$. This implies that 
    \[ \Xi(\overline{\varphi}^{\Lambda_1}(X \oplus Y)) = \overline{\varphi}^{\Lambda_2}(\Xi(X \oplus Y))\]
    in Case TF-2b. The same equality holds for all other cases from the formulas of \cref{thm:bkt_tf_formula}, since then $\Xi$ is simply the full inclusion of $\mods \Lambda_1$ into $\mods \Lambda_2$.
\end{proof}

We have collected all necessary intermediate results to prove the main result of this section.

\begin{proof}[Proof of \cref{thm:longerlengths}]
    As discussed, it is sufficient to consider $\Lambda_1$ and $\Lambda_2$ in \cref{set:sec5}. By \cref{lem:bijTFtt}, there is a bijection $\Xi$ from TF-ordered $\tau$-rigid modules in $\mods \Lambda_1$ to TF-ordered $\tau$-rigid modules in $\mods \Lambda_2$ which induces a bijection between $\tau$-exceptional sequences by \cref{thm:MendozaTreffinger}. Moreover, this restricts to bijections between TF-ordered $\tau$-tilting modules and complete $\tau$-exceptional sequences. In order to prove the theorem, we need to establish the equality
    \begin{equation}\label{eq:desired} \overline{\varphi}_i^{\Lambda_2}(\Xi(M)) = \Xi(\overline{\varphi_i}^{\Lambda_1}(M))\end{equation}
    for all TF-ordered $\tau$-tilting modules $M$ and all indices $1 \leq i \leq n-1$. We begin with the case $i=n-1$ and obtain
    \begin{align*}
       \Xi(\overline{\varphi}_{n-1}^{\Lambda_1}(M)) &=  \Xi( (f_{\overline{\varphi}^{\Lambda_1}(M_{n-1} \oplus M_n)}^{\Lambda_1})^{-1}(f_{M_{n-1} \oplus M_n}^{\Lambda_1}(M_1 \oplus \dots \oplus M_{n-2})) \oplus \overline{\varphi}^{\Lambda_1}(M_{n-1} \oplus M_n)) \\
       & = \Xi( (f_{\overline{\varphi}^{\Lambda_1}(M_{n-1} \oplus M_n)}^{\Lambda_1})^{-1}(f_{M_{n-1} \oplus M_n}^{\Lambda_1}(M_1 \oplus \dots \oplus M_{n-2}))) \oplus \Xi(\overline{\varphi}^{\Lambda_1}(M_{n-1} \oplus M_n)) \\
       & = (f_{\Xi(\overline{\varphi}^{\Lambda_1}(M_{n-1} \oplus M_n))}^{\Lambda_2})^{-1}(f_{\Xi(M_{n-1} \oplus M_n)}^{\Lambda_2}(\Xi(M_1 \oplus \dots \oplus M_{n-2}))) \oplus \Xi(\overline{\varphi}^{\Lambda_1}(M_{n-1} \oplus M_n)) \\
       &= (f_{\overline{\varphi}^{\Lambda_2}(\Xi(M_{n-1} \oplus M_n))}^{\Lambda_2})^{-1} ( f_{\Xi(M_{n-1} \oplus M_{n})}^{\Lambda_2}(\Xi(M_1 \oplus \dots \oplus M_{n-2}))) \oplus \overline{\varphi}^{\Lambda_2}(\Xi(M_{n-1} \oplus M_{n})) \\
       &= \overline{\varphi}_{n-1}^{\Lambda_2}(\Xi(M)) 
    \end{align*}
    where the first and last equalities follow from \cite[Thm. 3.31]{bkt}, the second equality follows from the additivity of $\Xi$, the third equality follows from \cref{lem:reductionstep}, because the torsion-free functor is additive, and the fourth equality follows from \cref{prop:mutationstep}.

    For the case $i = n-2$ we work in the $\tau$-perpendicular subcategories $J^{\Lambda_1}(M_n)$ and $J^{\Lambda_2}(\Xi(M_n))$ by \cite[Prop. 3.32]{bkt}. By \cref{eq:jassosplit}, these may be split into two components, one of which they share and the other which falls into \cref{set:sec5}. It is clear that if either $M_{n-2}$ or $M_{n-1}$ lies in the shared component, then the result is trivial. If they both lie in the component which contains $f_{M_n}^{\Lambda_1}(P_{\Lambda_1}(0))$, then we are mutating at the right-most entry, as above. This then yields
    \begin{align*}
       \Xi(\overline{\varphi}_{n-2}^{\Lambda_1}(M))
       &=  \Xi( ( f_{M_n}^{\Lambda_1})^{-1} (\overline{\varphi}^{J^{\Lambda_1}(M_n)}_{(n-1)-1} (f_{M_n}^{\Lambda_1}(M_1 \oplus \dots \oplus M_{n-1})))) \oplus \Xi(M_n) \\
       & = (f_{\Xi(M_n)}^{\Lambda_2})^{–1} ( \Xi^{M_n}(  \overline{\varphi}^{J^{\Lambda_1}(M_n)}_{(n-1)-1}(f_{M_n}^{\Lambda_1}(M_1 \oplus \dots \oplus M_{n-1})))) \oplus \Xi(M_n) \\
       &= (f_{\Xi(M_n)}^{\Lambda_2})^{–1}  ( \overline{\varphi}^{J^{\Lambda_2}(\Xi(M_n))}_{(n-1)-1}(\Xi^{M_n}(f_{M_n}^{\Lambda_1}(M_1 \oplus \dots M_{n-1})))) \oplus \Xi(M_n) \\
       &= (f_{\Xi(M_n)}^{\Lambda_2})^{–1}  ( \overline{\varphi}^{J^{\Lambda_2}(\Xi(M_n))}_{(n-1)-1}(f_{\Xi(M_n)}^{\Lambda_2}(\Xi(M_1 \oplus \dots M_{n-1})))) \oplus \Xi(M_n) \\
       &= \overline{\varphi}_{n-2}^{\Lambda_2}(\Xi(M)).
    \end{align*}
    where the first and final equations follow from \cite[Prop. 3.32]{bkt} and the additivity of $\Xi$, the second and fourth equations follow from \cref{lem:reductionbijcommutes} and finally the third equation is the above ``right-most'' $i = |J^{\Lambda_1}(X)|-1 = (n-1)-1$ case in $J^{\Lambda_1}(X)$, using the fact that $\Xi^{M_n}$ is a bijection from $J^{\Lambda_1}(X)$ to $J^{\Lambda_2}(\Xi(X))$ by \cref{lem:reducedbijection}. We can now repeat this process to obtain the desired \cref{eq:desired} for all $1 \leq i \leq n-1$ which completes the proof. This shows that the mutation of TF-ordered $\tau$-tilting modules, and by \cite[Def. 3.28, Prop. 3.32]{bkt} also the mutation of complete $\tau$-exceptional sequences, in $\mods \Lambda_1$ and $\mods \Lambda_2$ coincide. In particular, since the mutation of complete $\tau$-exceptional sequences of $C_n = \mathcal{N}(n,\dots, n)$ satisfies the braid group relations by \cref{thm:braid_group_relations_satisfied_on_cn}, so do the complete $\tau$-exceptional sequences in $\mods \Lambda_1$ and in $\mods \Lambda_2$.
\end{proof}

\section{Counterexamples for the remaining non-hereditary cases}
The goal of this section is to show that the classes of Nakayama algebras we have found in the previous two sections are the only non-hereditary examples for which the mutation of complete $\tau$-exceptional sequences satisfies the braid group relations.

\begin{theorem}\label{thm:notethm1}
    Let $\Lambda$ be a basic and connected Nakayama algebra. Assume that there is an indecomposable projective $\Lambda$-module $P$ with $\ell(P) < |\Lambda|$. Then the mutation of $\tau$-exceptional sequences satisfies the braid group relation \textnormal{B2} if and only if $\Lambda$ is hereditary. 
\end{theorem}

By \cite[Thm. V.3.2]{assem_skowronski_simson_2006}, any basic and connected Nakayama algebra may be expressed as $\Lambda \cong KQ/I$ with $Q$ being either $\overrightarrow{\A}_n$ or $\overrightarrow{\Delta}_n$ for some $n \geq 1$ and with $I$ being an admissible ideal, see \cref{sec:cycliccase}. It is immediately clear that $\Lambda$ is finite dimensional and hereditary if and only if $\Lambda = k \overrightarrow{\A}_n$ and in this case the mutation of ($\tau$)-exceptional sequences in $\mods \Lambda$ satisfies the braid relations by \cite{cbw92,ringel_exceptional}, see also \cite[Thm. 0.5]{BHM2024}. 

Throughout the rest of this section, let $\Lambda \cong KQ/I$ denote a basic and connected Nakayama algebra. For $|\Lambda|=2$, there is nothing to show, thus assume $|\Lambda| \geq 3$. We label the vertices of $\Lambda$ from $0$ to $|\Lambda|-1$ as in \cref{sec:cycliccase}. If the underlying quiver of $\Lambda$ is cyclic, then all indices are taken modulo $|\Lambda|$. 

\begin{lemma}\label{lem:specialindex}
    Assume that $\Lambda$ satisfies the assumptions of \cref{thm:notethm1} and that $\Lambda$ is not hereditary. Then there exists an index $i \in \{0, \dots, |\Lambda|-1\}$ such that $\ell(P(i+1)) < |\Lambda|$ and $\ell(P(i)) \geq \ell(P(i+1))$.
\end{lemma}
\begin{proof}
    Let $n \coloneqq |\Lambda|$ and let $R \in \mods \Lambda$ be an indecomposable projective $\Lambda$-module such that $\ell(R) < n$, which exists by assumption.
    
    Assume first that $\Lambda \cong K \overrightarrow{\A}_n /I$. In this case it is clear that $\ell(P(i)) \leq i+1 \leq n$ for all indices $i \in \{0,\dots, n-1\}$. Moreover, if $\Lambda$ is not hereditary, it follows that $\ell(P(n-1)) < n$. Let $j$ be the least positive integer such that $\rad^j P(n-1)$ is nonzero and nonprojective. Such an integer $j$ exists since the module $P(n-1)$ has at most $n-2$ proper nonzero submodules, at least one of which has to be nonprojective. If they were all projective, then in particular, the simple $\soc(P(n-1))$ would be projective, but $\Lambda$ only has one simple projective, which is $P(0)$, since it is connected. However, $\ell(P(n-1))<n$ and $|\Lambda|>2$ imply that $P(0) \not \cong \soc( P(n-1))$. Thus, consider $P(n-j-1)$, which is the projective generating $\rad^j P(n-1)$. It clearly satisfies 
    \[ \ell(P(n-j-1)) \geq \ell(\rad^{j} P(n-1)) +1 =\ell(\rad^{j-1} P(n-1)) = \ell(P(n-j)) \]
    since $\rad^{j-1} P(n-1) = P(n-j)$ by the minimality of $j$. 
    Consequently, setting $i=n-j-1$ gives the desired index in this case. By construction $j \geq 1$, so that $i \leq n-2$.

    Assume now that $\Lambda \cong K \overrightarrow{\Delta}_n /I$ and let $h \in \{0,\dots, n-1\}$ be the index such that $R \cong P(h)$. Again, let $j$ be the least integer such that $\rad^j P(h)$ is nonzero and nonprojective. Such an integer $j$ exists since there are no simple projective $\Lambda$-modules. Like above, consider $P(h-j)$, which is the projective generating $\rad^j P(h)$. It clearly satisfies 
    \[ \ell(P(h-j)) \geq \ell(\rad^j P(h)) +1 = \ell(\rad^{j-1} P(h)) = \ell(P(h-j+1)).\]
    Thus, setting $i=h-j$ gives the desired index in this case. This completes the proof.
    \end{proof}

    We note that if $\Lambda \cong k\overrightarrow{\A}_n/I$ in \cref{lem:specialindex}, then by construction $i\in \{1, \dots, |\Lambda|-2\}$. Thus when choosing $i$ as in \cref{lem:specialindex}, the module $P(i+1)$ is a well-defined projective nonsimple $\Lambda$-module for both cyclic and linear Nakayama algebras. The following module plays a central role in the proof of \cref{thm:notethm1}.

    \begin{lemma}\label{lem:specialtaurigid}
        Assume the setting of \cref{lem:specialindex} and let $i \in \{0, \dots, |\Lambda|-1\}$ be an index such that $\ell(P(i+1)) < |\Lambda|$ and $\ell(P(i)) \geq \ell(P(i+1))$. Then the ordered $\Lambda$-module $P(i+1) \oplus P(i) \oplus \rad P(i+1)$ is a TF-ordered $\tau$-rigid module.
    \end{lemma}
    \begin{proof}
        Since $\ell(P(i)) \geq \ell(P(i+1))$, the module $P(i)$ is not a submodule of $P(i+1)$. This means $\rad P(i+1) \not \cong P(i)$, in fact $\rad P(i+1)$ is a proper quotient of $P(i)$. It is nonzero since $P(i+1)$ is not simple. Define 
        \[ M= P(i+1) \oplus P(i) \oplus \rad P(i+1).\]
        It is clear that $\tau M \cong \tau(\rad P(i+1))$. Moreover, since $\ell(\rad P(i+1)) < \ell(P(i+1)) < |\Lambda|$, it follows from \cite[Lem. 2.5]{Adachi2016} that $\rad P(i+1)$ is $\tau$-rigid. Because of the uniseriality of all indecomposable $\Lambda$-modules, it suffices to show that 
        \[ \Hom(P(i+1), \tau (\rad P(i+1))=0 = \Hom(P(i), \tau(\rad P(i+1))\]
        in order to conclude that $M$ is $\tau$-rigid. It is then clearly TF-ordered. The equalities above follow from the simple observation that 
        \[ \ell(\tau(\rad P(i+1)) = \ell(\rad P(i+1)) \leq |\Lambda|-2\]
        and since $\topp(\tau(\rad P(i+1))) \cong S(i-1)$ by \cite[Prop. 4.6(7)(8)]{bkt}, neither $S(i+1)$ nor $S(i)$ is a composition factor of $\tau(\rad P(i+1))$. This concludes the proof.
    \end{proof}

    We are now able to proof the main result of this section, completing the classification of \cref{thm:introthm4} in the introduction.
    
    \begin{proof}[Proof of \cref{thm:notethm1}]
        If $\Lambda$ is hereditary, then the mutation of $\tau$-exceptional sequences in $\mods \Lambda$ coincides with the mutation of exceptional sequences in $\mods \Lambda$ by \cite[Thm. 0.5]{BHM2024}. It thus satisfies the braid group relations by \cite{cbw92,ringel_exceptional}. Conversely, assume that $\Lambda$ is not hereditary. By \cref{lem:specialindex}, there exists an index $i \in \{0, \dots, |\Lambda|-1\}$ such that $\ell(P(i+1)) < |\Lambda|$ and such that $\ell(P(i)) \geq \ell(P(i+1))$. By \cref{lem:specialtaurigid}, we may consider the TF-ordered $\tau$-rigid module
        \[ M \cong P(i+1) \oplus P(i) \oplus \rad P(i+1).\]
        We will show that the mutation $\overline{\varphi}$ of TF-ordered $\tau$-rigid modules applied to the module above does not satisfy the braid group relation B2. By \cite[Def. 3.28, Prop. 3.32]{bkt} this implies that the mutation of $\tau$-exceptional sequences does not satisfy the braid group relation B2. We make use of the explicit description of $\overline{\varphi}$ of \cref{thm:bkt_tf_formula}.

        In the following, we use the existence of an isomorphism $\rad P(i+1) \cong P(i)/\rad^j P(i)$ for some $1 \leq j < \ell(P(i))$. In particular, $j \leq  \ell(P(i))-1 \leq n-2$ means that $P(i-j) \not \cong P(i+1)$ and $P(i-j) \not \cong P(i)$. Thus,
        \begin{equation}\label{eq:fail1}
        \begin{aligned}
            \overline{\varphi}_2 \circ \overline{\varphi}_1 \circ \overline{\varphi}_2(M) &= \overline{\varphi}_2 \circ \overline{\varphi}_1(P(i+1) \oplus P(i-j) \oplus P(i)) &\text{(TF-2b in $\mods \Lambda$)} \\
            &= \overline{\varphi}_2(P(i-j) \oplus P(i+1) \oplus P(i)) &\text{(TF-1a in $J(P(i))$} \\
            &= P(i-j) \oplus P(i+1) \oplus S(i+1) &\text{(TF-1b in $\mods \Lambda$)}.
        \end{aligned}
        \end{equation}
        In the above, the module $P(i-j)$ is the projective generating $\rad^j P(i)$. For the first equality of \cref{eq:fail1}, it is sufficient to observe that $\rad P(i+1) = t_{P(i)}P(i+1)$ is the maximal proper subobject of $P(i+1)$ in order to conclude that $f_{P(i-j) \oplus P(i)}^{-1}(f_{P(i)} P(i+1)) = P(i+1)$. For the second equality, we note that $f_{P(i)} P(i+1) \cong S(i+1)$ is a simple projective in $J(P(i))$, and hence clearly $\Hom_{J(P(i))}(f_{P(i)}P(i-j), S(i+1)) = 0$. Finally, for the third equality of \cref{eq:fail1} we again use the isomorphism $f_{P(i)}P(i+1) \cong S(i+1)$ and observe that due to the ordering of the vertices $i$, $i+1$ and $i-j$ of $\Lambda$ we obtain an isomorphism $f_{P(i+1)}^{-1}(f_{P(i+1)\oplus P(i)} P(i-j)) \cong P(i-j)$. 

        On the other hand, consider the following sequence of mutations.
        \begin{equation}\label{eq:fail2}
        \begin{aligned}
            \overline{\varphi}_1 \circ \overline{\varphi}_2 \circ \overline{\varphi}_1(M) &= \overline{\varphi}_1 \circ  \overline{\varphi}_2(P(i) \oplus P(i+1) \oplus \rad P(i+1)) &\text{(TF-1a in $J(\rad P(i+1))$)} \\
            &= \overline{\varphi}_1( P(i) \oplus \rad P(i+1) \oplus P(i+1)) &\text{(TF-3 in $\mods \Lambda$)}.
        \end{aligned}
        \end{equation}
    The first line of \cref{eq:fail2} follows from the observation that $f_{\rad P(i+1)} P(i+1) \cong S(i+1)$. Similar to above, the module $S(i+1)$ is a simple projective in $J(\rad P(i+1))$ and hence 
    \[ \Hom_{J(\rad P(i+1))}(f_{\rad P(i+1)}P(i), S(i+1))=0. \]
    The second mutation of \cref{eq:fail2} falls into case TF-3 since $\rad P(i+1)$ is not projective, the Bongartz completion of $P(i+1)$ consists of the sum of all projective $\Lambda$-modules and the co-Bongartz completion of $P(i+1)$ consists of its quotient modules by \cite[Prop. 4.11]{bkt}. This explains the second equality in \cref{eq:fail2}. It is not necessary to compute the final mutation, since the right-most entries of $\overline{\varphi}_1 \circ \overline{\varphi}_2 \circ \overline{\varphi}_1(M)$ and $\overline{\varphi}_2 \circ \overline{\varphi}_1 \circ \overline{\varphi}_2(M)$ do not coincide as $S(i+1) \not \cong P(i+1)$ and applying $\overline{\varphi}_1$ does not affect this entry.

    In conclusion, the mutation of TF-ordered $\tau$-rigid modules in $\mods \Lambda$ does not satisfy the braid group relation B2. Thus \cite[Def. 3.28, Prop. 3.32]{bkt} yields that the mutation of $\tau$-exceptional sequences in $\mods \Lambda$ does not satisfy the braid group relation B2.
    \end{proof}

\printbibliography

\end{document}